\documentclass[reqno]{amsart} 
\usepackage{enumerate,amsmath,amssymb} 
\usepackage{pstricks,pst-node,pst-coil,pst-plot} 
\usepackage{hyperref}
\theoremstyle{plain}

\newtheorem{thm}{Theorem}[section]
\newtheorem{theorem}[thm]{Theorem}

\newtheorem{lemma}[thm]{Lemma}

\newtheorem{proposition}[thm]{Proposition}

\theoremstyle{remark}

\newtheorem{remark}[thm]{Remark}

\theoremstyle{definition}

\newtheorem{definition}[thm]{Definition}

\newcounter{mnotecount}[section]

\newcommand{\definedas}{\mathrel{\raise.095ex\hbox{\rm :}\mkern-5.2mu=}}

\def\epsilon{{\varepsilon}}
\def\phi{{\varphi}}

\let\<\langle 
\let\>\rangle

\newcommand{\bR}{\mathbb{R}}

\newcommand{\bH}{\mathbb{H}}

\newcommand{\cN}{\mathcal{N}}

\renewcommand{\hbar}{\overline{h}}
\newcommand{\Gambar}{\overline{\Gamma}}
\newcommand{\Gamtil}{\widetilde{\Gamma}}

\newcommand{\Htil}{\widetilde{H}}
\newcommand{\Hbar}{\overline{H}}

\newcommand{\nablabar}{\overline{\nabla}}
\newcommand{\bbar}{\overline{b}}

\newcommand{\btil}{\widetilde{b}}

\DeclareMathOperator{\diag}{diag}

\DeclareMathOperator{\tr}{tr}
\DeclareMathOperator{\divg}{div}

\newcommand{\pdiff} [2]{\frac{\partial #1}{\partial #2}}

\newcommand{\scal}{\mathrm{Scal}}

\DeclareMathOperator{\hess}{Hess}
\newcommand{\hessdd}[2]{\nabla^2_{#1, #2}}

\newcommand{\hessbardd}[2]{\overline{\nabla}^2_{#1, #2}}
\newcommand{\hesstildd}[2]{\widetilde{\nabla}^2_{#1, #2}}

\newcommand{\sff}{S}

\newcommand{\sffdd}[2]{\sff_{#1 #2}}

\newcommand{\sffr}{\widetilde{\sff}}
\newcommand{\sffrdd}[2]{\sffr_{#1 #2}}

\newcommand{\sffb}{\overline{\sff}}
\newcommand{\sffbdd}[2]{\sffb_{#1 #2}}

\def\epsilon{{\varepsilon}}

\def\phi{{\varphi}}

\let\<\langle 
\let\>\rangle

\begin{document} 
%%%%%%%%%%%%%%%%%%%%%%%%%%%%%%%%%%%%%%%%%%%%%%%%%%%%%%%%%%%%%%%%%%%%%%%%%

%%%%%%%%%%%%%%%%%%%%%%%%%%%%%%%%%%%%%%%%%%%%%%%%%%%%%%%%%%%%%%%%%%%%%%%%%
%\begin{center}
%\framebox{\framebox{
%\vbox{This is project {\red \Project}\\
%Current version {\blue\Version}, from
%{\blue\Datum}, most recent changes by {\blue\Person}.}
%}}
%\end{center}
%%%%%%%%%%%%%%%%%%%%%%%%%%%%%%%%%%%%%%%%%%%%%%%%%%%%%%%%%%%%%%%%%%%%%%%%%

\title
{Penrose type inequalities for asymptotically hyperbolic graphs}
 
\author{Mattias Dahl}
\address{Institutionen f\"or Matematik \\
 Kungliga Tekniska H\"ogskolan \\
 100 44 Stockholm \\
 Sweden} \email{dahl@math.kth.se}

\author{Romain Gicquaud}
\address{Laboratoire de Math\'ematiques et de Physique Th\'eorique \\
 UFR Sciences et Technologie \\
 Facult\'e Fran\c cois Rabelais \\
 Parc de Grandmont \\
 37200 Tours \\
 France} \email{romain.gicquaud@lmpt.univ-tours.fr}

\author{Anna Sakovich}
\address{Institutionen f\"or Matematik \\
 Kungliga Tekniska H\"ogskolan \\
 100 44 Stockholm \\
 Sweden} \email{sakovich@math.kth.se}

\begin{abstract}
In this paper we study asymptotically hyperbolic manifolds given 
as graphs of asymptotically constant functions over hyperbolic space 
$\bH^n$. The graphs are considered as unbounded hypersurfaces of 
$\bH^{n+1}$ which carry the induced metric and have an interior 
boundary. For such manifolds the scalar curvature appears in the 
divergence of a 1-form involving the integrand for the asymptotically 
hyperbolic mass. Integrating this divergence we estimate the mass by 
an integral over the inner boundary. In case the inner boundary 
satisfies a convexity condition this can in turn be estimated in terms 
of the area of the inner boundary. The resulting estimates are similar 
to the conjectured Penrose inequality for asymptotically hyperbolic 
manifolds. The work presented here is inspired by Lam's article 
\cite{LamGraph} concerning the asymptotically Euclidean case.
Using ideas developed by Huang and Wu in \cite{HuangWu2} we can in 
certain cases prove that equality is only attained for the anti-de 
Sitter Schwarzschild metric.
\end{abstract}

% 53C21 Methods of Riemannian geometry, including PDE methods; 
% curvature restrictions
%
% 83C05 Einstein's equations (general structure, canonical formalism, 
% Cauchy problems) 
%
% 83C30 View Publications (1980-now) Asymptotic procedures (radiation, 
% news functions, H-spaces, etc.) 

\date{\today}


\maketitle

\tableofcontents

%%%%%%%%%%%%%%%%%%%%%%%%%%%%%%%%%%%%%%%%%%%%%%%%%%%%%%%%%%%%%%%%%%%%%%%%%
\section{Introduction}
%%%%%%%%%%%%%%%%%%%%%%%%%%%%%%%%%%%%%%%%%%%%%%%%%%%%%%%%%%%%%%%%%%%%%%%%%

In 1973, R. Penrose conjectured that the total mass of a space-time
containing black holes cannot be less than a certain function of the sum
of the areas of the event horizons. Black holes are objects whose
definition requires knowledge of the global space-time. Hence, given
Cauchy data (which are the only data needed to define the total mass
of a space-time), finding event horizons would require solving the
Einstein equations. As a consequence, in the current formulation of the
Penrose conjecture, event horizons are usually replaced by the weaker
notion of apparent horizons. We refer the reader to
\cite[Chapter XIII]{ChoquetBruhat} for further details.

The classical Penrose conjecture takes the following form: Let
$(M, g, k)$ be Cauchy data for the Einstein equations, that is a triple
where $(M, g)$ is a Riemannian 3-manifold and $k$ is a symmetric 2-tensor
on $M$. Assume that $(M, g, k)$ satisfies the dominant energy condition
\[
\mu \geq |J|,
\] 
where $\mu$ and $J$ are defined through
\[\left\lbrace
\begin{aligned}
 \mu &\definedas \frac{1}{2} \left(\scal^g - |k|^2_g + (\tr_g k)^2\right),\\
 J &\definedas \divg(k) - d(\tr_g k).
\end{aligned}
\right.\]
Assume further that $(M, g, k)$ is asymptotically Euclidean. A compact
oriented surface $\Sigma \subset M$ is called an apparent horizon if
$\Sigma$ satisfies \[H^g + \tr^\Sigma k = 0,\] where $H^g$ is the trace
of the second fundamental form of $\Sigma$ computed with respect to 
the outgoing normal $\nu$ of $\Sigma$, that is $\sff(X, Y) = 
\<\nabla_X \nu, Y\>$ for any vectors $X$ and $Y$ tangent to $\Sigma$,  
and $\tr^\Sigma k$ is the trace of $k$ restricted to the tangent space 
of $\Sigma$ for the metric induced by $g$. Hence viewing $(M, g, k)$ 
as immersed in a space-time, the expansion of $\Sigma$ in the future 
outgoing light-like direction vanishes. We assume that $\Sigma$ is 
outermost, that is $\Sigma$ contains all other apparent horizons in 
its interior. Note that $\Sigma$ may be disconnected. See 
\cite{AnderssonMetzger} for further details. Then the Penrose conjecture 
takes the form
\[
m \geq \sqrt{\frac{|\Sigma|}{16\pi}},
\] 
where $|\Sigma|$ denotes the area of $\Sigma$ and $m$ is the mass of the 
manifold $(M, g)$. Further, equality should hold only if the exterior of 
$\Sigma$ is isometric to a hypersurface in the exterior region of a 
Schwarzschild black hole with $k$ equal to the second fundamental form of 
this hypersurface.

This conjecture can be generalized to higher dimensional manifolds. All
the statements are the same except for the inequality which in $n$ 
dimensions reads
\[
m \geq
\frac{1}{2} \left(\frac{|\Sigma|}{\omega_{n-1}}\right)^{\frac{n-2}{n-1}},
\]
where $\omega_{n-1}$ is the volume of the unit $(n-1)$-sphere.

Two major breakthroughs in the proof of this inequality were obtained 
almost simultaneously by Huisken, Ilmanen \cite{HuiskenIlmanen} and 
Bray \cite{BrayPenrose} for 3-manifolds. They both deal with the 
time-symmetric case, i.e. when $k=0$. The result of Bray was extended to 
higher dimensions in \cite{BrayLee}. We refer the reader to the excellent 
reviews \cite{MarsPenrose} and \cite{BrayChrusciel} for further details. 
Recently, Lam \cite{LamGraph} gave a simple proof of the time-symmetric 
Penrose inequality for a manifold which is a graph of a smooth function 
over $\bR^n$. His proof was extended by Huang and Wu in \cite{HuangWu1} to
give a proof of the positive mass theorem (including the rigidity statement)
for asymptotically Euclidean manifolds which are submanifolds of $\bR^{n+1}$.
More general ambient spaces were considered by de Lima and Gir\~ao in
\cite{LimaGirao1}.

The Penrose conjecture can be generalized to space-times with negative 
cosmological constant. Up to rescaling, we can assume that the cosmological 
constant $\Lambda$ equals $-\frac{n(n-1)}{2}$. Restricting ourselves to the 
time-symmetric case, the dominant energy condition then reads 
\[
\scal^g \geq -n(n-1).
\] 
The lower bound for the mass (defined in Section \ref{secAH}) is then 
conjectured to be given by the mass of the anti-de Sitter Schwarzschild 
space-time (see Section \ref{secAdS}),
\begin{equation} \label{AH-Penrose}
m \geq 
\frac{1}{2}\left[
\left( \frac{|\Sigma|}{\omega_{n-1}} \right)^{\frac{n-2}{n-1}} 
+ \left(\frac{|\Sigma|}{\omega_{n-1}}\right)^{\frac{n}{n-1}}
\right].
\end{equation}

In this paper, we prove weaker forms of this inequality for manifolds
which are graphs over the hyperbolic space $\bH^n$ when we endow the
manifold $\bH^n \times \bR$ with a certain hyperbolic metric. See
Theorem \ref{THM-Main}.

After the first version of this article appeared on arXiv,
de Lima and Gir\~ao posted an article dealing with another case
of the asymptotically hyperbolic Penrose inequality \cite{LimaGirao2}.
Rigidity was addressed by de Lima and Gir\~ao in \cite{LimaGirao3} 
and by Huang and Wu in \cite{HuangWu2}. The approach used in 
\cite{HuangWu2} does not require any further assumption and we shall
extend it to our context in Section \ref{secRigidity}.

The outline of this paper is as follows. In Section \ref{secAH}, we
define the mass of a general asymptotically hyperbolic manifold.
We explicit the anti-de Sitter Schwarzschild metric in Section \ref{secAdS}.
In Section \ref{secScal} we prove that the scalar curvature of a graph 
has divergence form (Equation \eqref{eqScalCurvDivg}) and that its 
integral is related to the mass (Lemma \ref{lmIntegrCourbScal}). 
In Section \ref{secPenrose}, we prove the first part of Theorem
\ref{THM-Main}. Rigidity is addressed in Section \ref{secRigidity}.

%%%%%%%%%%%%%%%%%%%%%%%%%%%%%%%%%%%%%%%%%%%%%%%%%%%%%%%%%%%%%%%%%%%%%%%%%
\subsection*{Acknowledgements}
%%%%%%%%%%%%%%%%%%%%%%%%%%%%%%%%%%%%%%%%%%%%%%%%%%%%%%%%%%%%%%%%%%%%%%%%%

We thank Julien Cortier and Hubert Bray for helpful conversations.
We are also grateful to Gerhard Huisken for enlightening discussions on
the Aleksandrov-Fenchel inequalities and to Lan Hsuan-Huang for pointing
us to the article \cite{HuangWu2}.
Further, we want to give a special thanks to Christophe Chalons and 
Jean-Louis Tu who helped us with the proof of the results stated in 
Appendix~\ref{appopensets}. 

%%%%%%%%%%%%%%%%%%%%%%%%%%%%%%%%%%%%%%%%%%%%%%%%%%%%%%%%%%%%%%%%%%%%%%%%%
\subsection*{A note}
%%%%%%%%%%%%%%%%%%%%%%%%%%%%%%%%%%%%%%%%%%%%%%%%%%%%%%%%%%%%%%%%%%%%%%%%%

After this paper was finished the articles \cite{LimaGirao4} by de Lima 
and Gir\~ao, and \cite{BrHuWa} by Brendle, Hung, and Wang appeared on 
arXiv. In the first of these papers an Alexandrov-Fenchel type inequality 
for hypersurfaces in hyperbolic space is stated, which together with 
Proposition \ref{prop_penrose_graph} implies the Penrose inequality 
\eqref{AH-Penrose} for graphs. Certain steps of the proof seem to need 
further clarification, for example the convergence of hypersurfaces to 
round spheres under the inverse mean curvature flow. However, combining 
with arguments of the second paper \cite{BrHuWa} the result should follow.
Note also that a special case of \cite[Theorem 2]{BrHuWa} follows from our 
formula \eqref{eqIntHV} in Section \ref{changetoeucl}.

%%%%%%%%%%%%%%%%%%%%%%%%%%%%%%%%%%%%%%%%%%%%%%%%%%%%%%%%%%%%%%%%%%%%%%%%%
\section{Preliminaries}
%%%%%%%%%%%%%%%%%%%%%%%%%%%%%%%%%%%%%%%%%%%%%%%%%%%%%%%%%%%%%%%%%%%%%%%%%

%%%%%%%%%%%%%%%%%%%%%%%%%%%%%%%%%%%%%%%%%%%%%%%%%%%%%%%%%%%%%%%%%%%%%%%%%
\subsection{Asymptotically hyperbolic manifolds and the mass}
\label{secAH}
%%%%%%%%%%%%%%%%%%%%%%%%%%%%%%%%%%%%%%%%%%%%%%%%%%%%%%%%%%%%%%%%%%%%%%%%%

We define the mass of an asymptotically hyperbolic manifold following
Chru{\'s}ciel and Herzlich, see \cite{ChruscielHerzlich} and 
\cite{HerzlichMassFormulae}. In the special case of conformally compact 
manifolds this definition coincides 
with the definition given by Wang in \cite{WangMass}. In what follows, 
$n$-dimensional hyperbolic space is denoted by $\bH^n$ and its metric is 
denoted by $b$. In polar coordinates $b = dr^2 + \sinh^2 r \sigma$ 
where $\sigma$ is the standard round metric on $S^{n-1}$. 

Set 
$\cN \definedas \{ V \in C^{\infty}(\bH^n) \mid \hess^b V = V b \}$.
This is a vector space with a basis of the functions 
\[
V_{(0)} = \cosh r, \, 
V_{(1)} = x^1 \sinh r, \dots , \,
V_{(n)} = x^n \sinh r,
\]
where $x^1, \dots, x^n$ are the coordinate functions
on $\bR^{n}$ restricted to $S^{n-1}$. If we consider $\bH^n$ as the upper 
unit hyperboloid in Minkowski space $\bR^{n,1}$ then the functions $V_{(i)}$ are the 
restrictions to $\bH^n$ of the coordinate functions of $\bR^{n,1}$. 
The vector space $\cN$ is equipped with a Lorentzian inner product 
$\eta$ characterized by the condition that the basis above is orthonormal, 
$\eta(V_{(0)}, V_{(0)}) = 1$, and $\eta(V_{(i)}, V_{(i)}) = -1$ for 
$i=1,\dots,n$. We also give $\cN$ a time orientation by specifying that 
$V_{(0)}$ is future directed. The subset $\cN^+$ of positive functions then 
coincides with the interior of the future lightcone. Further we denote by 
$\cN^1$ the subset of $\cN^+$ of functions $V$ with $\eta(V,V) = 1$, this 
is the unit hyperboloid in the future lightcone of $\cN$. All 
$V \in \cN^1$ can be constructed as follows. Choose an arbitrary point 
$p_0 \in \bH^n$. Then the function
\[
V \definedas \cosh d_b(p, \cdot)
\] 
is in $\cN^1$.

A Riemannian manifold $(M,g)$ is said to 
be {\em asymptotically hyperbolic} if there exist a compact subset and a
diffeomorphism at infinity $\Phi : M \setminus K \to \bH^n \setminus B$,
where $B$ is a closed ball in $\bH^n$, for which $\Phi_* g$ and $b$ are
uniformly equivalent on $\bH^n \setminus B$ and
\begin{subequations}
\begin{equation} \label{decay1}
\int_{\bH^n \setminus B} 
\left( 
| e |^2 + | \nabla^b e |^2
\right) \cosh r \, d\mu^b < \infty, 
\end{equation}
\begin{equation} \label{decay2}
\int_{\bH^n \setminus B} |\scal ^g + n(n-1)| \cosh r \, d\mu^b < \infty, 
\end{equation}
\end{subequations}
where $e \definedas \Phi_* g - b$ and $r$ is the (hyperbolic) distance
from an arbitrary given point in $\bH^n$.

The {\em mass functional} of $(M,g)$ with respect 
to $\Phi$ is the functional on $\cN$ defined by
\[
H_{\Phi} (V)
=
\lim_{r \to \infty} \int_{S_r} \left(
V (\operatorname{div}^b e- d \tr^b e) + (\tr^b e) dV - e(\nabla^b V, \cdot)
\right) (\nu_r) \, d \mu^b
\] 
Proposition~2.2 of \cite{ChruscielHerzlich} tells us that these limits 
exist and are finite under the asymptotic decay conditions 
\eqref{decay1}-\eqref{decay2}. If $\Phi$ is a chart at infinity as above 
and $A$ is an isometry of the hyperbolic metric $b$ then $A \circ \Phi$ 
is again such a chart and it is not complicated to verify that
\[
H_{A \circ \Phi} (V) = H_{\Phi} (V \circ A^{-1}) .
\]
If $\Phi_1$, $\Phi_2$ are charts at infinity as above, then 
from \cite[Theorem~2.3]{HerzlichMassFormulae}
we know that there is an isometry $A$ of $b$ so that 
$\Phi_2 = A \circ \Phi_1$ modulo lower order terms which do not 
affect the mass functional.

The mass functional $H_{\Phi}$ is timelike future directed if 
$H_{\Phi}(V) > 0$ for all $V \in \cN^+$. In this case the {\em mass} 
of the asymptotically hyperbolic manifold $(M,g)$ is defined by 
\[
m \definedas
\frac{1}{2(n-1)\omega_{n-1}} \inf_{\cN^1} H_{\Phi}(V).
\]
Further if $H_{\Phi}$ is timelike future directed we may replace
$\Phi$ by $A \circ \Phi$ for a suitably chosen isometry $A$ so that
$m = H_{\Phi} (V_{(0)})$. Coordinates with this property are called 
{\em balanced}.

The positive mass theorem for asymptotically hyperbolic manifolds, 
\cite[Theorem~4.1]{ChruscielHerzlich} and \cite[Theorem~1.1]{WangMass}, 
states that the mass functional is timelike future directed or zero 
for complete asymptotically hyperbolic spin manifolds with scalar 
curvature $\scal \geq -n(n-1)$. In 
\cite[Theorem~1.3]{AnderssonCaiGalloway} the same result is proved 
with the spin assumption replaced by assumptions on the dimension 
and on the geometry at infinity.

%%%%%%%%%%%%%%%%%%%%%%%%%%%%%%%%%%%%%%%%%%%%%%%%%%%%%%%%%%%%%%%%%%%%%%%%%
\subsection{Asymptotically hyperbolic graphs}
\label{secAHgraphs}
%%%%%%%%%%%%%%%%%%%%%%%%%%%%%%%%%%%%%%%%%%%%%%%%%%%%%%%%%%%%%%%%%%%%%%%%%

The purpose of this paper is to prove versions of the Riemannian 
Penrose inequality for an asymptotically hyperbolic graph over the 
hyperbolic space $\bH^n$. We consider such a graph as a submanifold of 
$\bH^{n+1}$. In what follows we will denote tensors associated to 
$\bH^{n+1}$ with a bar. In particular $\bbar$ will denote the hyperbolic 
metric on $\bH^{n+1}$. 

To shorten notation we now fix 
\[
V = V_{(0)} = \cosh r
\]
for the rest of the paper. As a model of $\bH^{n+1}$ we take 
$\bH^n \times \bR$ equipped with the metric
\[
\bbar \definedas b + V^2 ds \otimes ds
\] 

%We first construct a suitable model for $\bH^{n+1}$.
%
%\begin{proposition}\label{propModelHyp}
%Let $V: \bH^n \to \bR$ be a smooth function. The metric 
%\[
%\bbar 
%= \bbar^V
%= b + V^2 ds^2
%\] 
%on $\bH^n \times \bR$ has constant sectional curvature $-1$ if and
%only if $V \in \cN^+$, that is $V$ is a positive function with
%$\hess^b V = V b$. Up to a linear redefinition of $s$ we can assume 
%that $\eta(V,V) = 1$, that is $V \in \cN^1$.
%\end{proposition}

%The proof of Proposition \ref{propModelHyp} then follows from the
%well-known formula
%\[
%\riembaruddd{\mu}{\nu}{\sigma}{\tau} 
%= 
%\partial_\sigma \Gambar^\mu_{\nu\tau}
%- \partial_\tau \Gambar^\mu_{\nu\sigma}
%+ \Gambar_{\nu\tau}^\alpha \Gambar_{\alpha\sigma}^\mu
%- \Gambar_{\nu\sigma}^\alpha \Gambar_{\alpha\tau}^\mu.
%\]

Let $\Omega$ be a relatively compact open subset and let
$f : \bH^n \setminus \Omega \to \bR$ be a continuous function which is 
smooth on $\bH^n \setminus \overline{\Omega}$. We consider the graph
\[
\Sigma \definedas 
\{(x, s) \in \bH^n \times \bR \mid f(x) = s \}.
\]
Define the diffeomorphism $\Phi: \Sigma \to \bH^n \setminus \Omega$ 
by $\Phi^{-1}(p) = (p, f(p))$. The push-forward of the metric induced 
on $\Sigma$ is 
\[
g \definedas \Phi_* \bbar 
= (\Phi^{-1})^* \bbar 
= b + V^2 df \otimes df.
\]

We will study the situation when the graph $\Sigma$ is asymptotically 
hyperbolic with respect to the chart $\Phi$, that is when 
\[
e = V^2 df \otimes df
\]
satisfies \eqref{decay1}-\eqref{decay2} and
\begin{equation}\label{decay3}
 |e| = V^2 |df|^2 \to 0\text{ at infinity.}
\end{equation}
Note that Condition \eqref{decay1} is a consequence of the following
condition:
\[
\int_{\bH^n \setminus B}
 \left( \left|df\right|^4 + \left|\hess f\right|^4 \right) \cosh^5 r \, d\mu^b
 < \infty,
\]
that is to say that $df$ belongs to a certain weighted Sobolev space.

If this holds we say that $f$ is an {\em asymptotically hyperbolic
function} and $\Sigma$ is an {\em asymptotically hyperbolic graph}.
We define $f$ to be {\em balanced at infinity} if $\Phi$ are balanced
coordinates at infinity. In this case the mass of $\Sigma$ is given by 
$m = H_{\Phi} (V)$ with $V = V_{(0)}$.

In this paper we will prove the following theorem which gives 
estimates similar to the Penrose inequality for asymptotically 
hyperbolic graphs. In certain cases this theorem also describes the
situation when equality is attained. For exact formulations see 
Theorem~\ref{THM-AH-Penrose-Graph-HS},
Theorem~\ref{THM-AH-Penrose-Graph}, and
Theorem~\ref{thm_rigidity}.

\begin{theorem}\label{THM-Main}
Let $\Omega \subset \bH^n$ be a relatively compact open subset of
$\bH^n$ with smooth boundary. Assume that $\Omega$ contains an inner ball 
centered at the origin of radius $r_0$. 
Let $f : \bH^n \setminus \Omega \to \bR$ be an asymptotically 
hyperbolic function which is balanced at infinity. Assume that $f$ is 
locally constant on $\partial \Omega$ and that $|df| \to \infty$ at 
$\partial \Omega$ so that $\partial \Omega$ is a horizon ($H^g = 0$).
Further assume that the scalar curvature of the graph of $f$ satisfies
$\scal \geq -n(n-1)$. Then the mass $m$ of the graph is bounded from 
below as follows.
\begin{itemize}
\item
If $\partial \Omega$ has non-negative mean curvature with respect to
the metric $b$, $H \geq 0$, we have 
\begin{equation*} 
m \geq 
\frac{n-2}{2^n (n-1) n^{\frac{n}{n-1}}} V(r_0)
\left( \frac{|\partial \Omega|}{\omega_{n-1}} \right)^{\frac{n-2}{n-1}}
\end{equation*}
and
\begin{equation*}
m \geq \frac{1}{2} V(r_0) \frac{|\partial \Omega|}{\omega_{n-1}}.
\end{equation*}
\item
If $\Omega$ is an h-convex subset of $\bH^n$ we have
\begin{equation*}
m \geq 
\frac{1}{2} \left[
\left( \frac{|\partial \Omega|}{\omega_{n-1}} \right)^{\frac{n-2}{n-1}} 
+ \sinh r_0 \frac{|\partial \Omega|}{\omega_{n-1}}\right].
\end{equation*}
If equality holds and $df(\eta)(x) \to +\infty$ as 
$x \to \partial \Omega$ where $\eta$ is the outward normal of the level
sets of $f$ then the graph of $f$ is isometric to the $t=0$ slice of 
the anti-de Sitter Schwarzschild space-time.
\end{itemize}
\end{theorem}
Note that since $f$ is locally constant on $\partial \Omega$, the areas
of $\partial \Omega$ computed using the metric $b$ and using the metric
induced on the graph are equal.

%%%%%%%%%%%%%%%%%%%%%%%%%%%%%%%%%%%%%%%%%%%%%%%%%%%%%%%%%%%%%%%%%%%%%%%%%
\subsection{The anti-de Sitter Schwarzschild space-time}
\label{secAdS}
%%%%%%%%%%%%%%%%%%%%%%%%%%%%%%%%%%%%%%%%%%%%%%%%%%%%%%%%%%%%%%%%%%%%%%%%%

We remind the reader that the metric outside the horizon of the 
anti-de Sitter-Schwarzschild space in (spatial) dimension $n \geq 3$ and 
of mass $m \geq 0$ is given by
\begin{equation*}
\gamma_{\text{\rm AdS-Schw}} = 
- \left(1+\rho^2 - \frac{2m}{\rho^{n-2}}\right) dt^2
+ \frac{d\rho^2}{1+\rho^2 - \frac{2m}{\rho^{n-2}}} + \rho^2 \sigma,
\end{equation*}
where $\sigma$ is the standard round metric on the sphere $S^{n-1}$.
See for example \cite[Section 6]{MarsPenrose}. The horizon is the 
surface $\rho = \rho_0(m)$, where $\rho_0=\rho_0(m)$ is the unique solution of 
\[
1 + \rho^2 - \frac{2m}{\rho^{n-2}} = 0.
\]
Its area is given by $A_m = \omega_{n-1} \rho_0^{n-1}$, hence multiplying
the previous formula by $\rho_0^{n-2}$, we get
\[\begin{aligned}
m
&= \frac{1}{2} \left[ \rho_0^{n-2} + \rho_0^{n} \right]\\
&= \frac{1}{2} \left[
 \left(\frac{A_m}{\omega_{n-1}}\right)^{\frac{n-2}{n-1}}
+ \left(\frac{A_m}{\omega_{n-1}}\right)^{\frac{n}{n-1}} \right].
\end{aligned}\]
This justifies the form of the right-hand side of \eqref{AH-Penrose}.

Restricting to the slice $t=0$, we get the following Riemannian metric.
\begin{equation} \label{eqAdSSchSpatialArea}
g_{\text{\rm AdS-Schw}} = 
\frac{d\rho^2}{1+\rho^2 - \frac{2m}{\rho^{n-2}}} + \rho^2 \sigma.
\end{equation}

We want to explicit the spatial metric \eqref{eqAdSSchSpatialArea}
as the induced metric of a graph $\Sigma_{\text{\rm AdS-Schw}}$. By 
rotational symmetry, we choose the point $\rho=0$ as the origin 
and $f = f(\rho)$. In this coordinate system, the reference hyperbolic 
metric $b$ is given by 
\[
b = \frac{d\rho^2}{1 + \rho^2} + \rho^2 \sigma.
\] 
The function $V$ is given by $V = \sqrt{1+\rho^2}$. Hence we seek a 
function $f$ satisfying
\[
V^2 \left(\pdiff{f}{\rho}\right)^2 =
\frac{1}{1+\rho^2 - \frac{2m}{\rho^{n-2}}} - \frac{1}{1+\rho^2}.
\]
Note that when $\rho$ is close to $\rho_0$, this forces
$\pdiff{f}{\rho} = O((\rho-\rho_0)^{-\frac{1}{2}})$. Hence we can set
\begin{equation}
\label{eqHeightAdSSchw}
f(\rho) = 
\int_{\rho_0}^{\rho} \frac{1}{\sqrt{1+s^2}}
\sqrt{\frac{1}{1+s^2 - \frac{2m}{s^{n-2}}} - \frac{1}{1+s^2}} 
\, ds.
\end{equation}
Similarly, when $\rho \to \infty$, $f$ converges to a constant. This
contrasts with the Euclidean case where %in dimension $3$
the Schwarzschild metric written as a graph is a half parabola in any 
radial direction, see \cite{LamGraph}.

%%%%%%%%%%%%%%%%%%%%%%%%%%%%%%%%%%%%%%%%%%%%%%%%%%%%%%%%%%%%%%%%%%%%%%%%%
\section{Scalar curvature of graphs in $\bH^{n+1}$}\label{secScal}
%%%%%%%%%%%%%%%%%%%%%%%%%%%%%%%%%%%%%%%%%%%%%%%%%%%%%%%%%%%%%%%%%%%%%%%%%

%%%%%%%%%%%%%%%%%%%%%%%%%%%%%%%%%%%%%%%%%%%%%%%%%%%%%%%%%%%%%%%%%%%%%%%%%
\subsection{Computation of scalar curvature}
\label{subsection_comp_scal}
%%%%%%%%%%%%%%%%%%%%%%%%%%%%%%%%%%%%%%%%%%%%%%%%%%%%%%%%%%%%%%%%%%%%%%%%%

Let $f : \bH^n \setminus \overline{\Omega} \to \bR$ be a smooth function.
Recall that we defined its graph as
\[
\Sigma \definedas 
\{(x, s) \in \bH^n \times \bR \mid f(x) = s \} = F^{-1}(0),
\]
where $F(x, s) \definedas f(x) - s$. For vector fields $X$ and $Y$ on 
$\bH^n$ the vector fields $\overline{X} = X + \nabla_X f \partial_0$ and
$\overline{Y} = Y + \nabla_Y f \partial_0$ are tangent to $\Sigma$. We
use coordinates on $\bH^n$ to parametrize $\Sigma$. 

Recall that we identify $\bH^{n+1}$ with $\bH^n \times \bR$ with the 
metric $\bbar = b + V^2 ds\otimes ds$. When using coordinate notation,
latin indices $i, j, \ldots \in \{1, \ldots, n\}$ denote (any)
coordinates on $\bH^n$ while a zero index denotes the $s$-coordinate
on $\bR$. Greek indices go from $0$ to $n$, hence denote coordinates
on $\bH^{n+1}$. The Christoffel symbols of $\bbar$ are collected in the
following Lemma.
\begin{lemma}
\[\left\lbrace\begin{aligned}
\Gambar^0_{00} & = 0 \\
\Gambar^i_{00} & = -V \nabla^i V \\
\Gambar^0_{i0} & = \frac{\nabla_i V}{V} \\
\Gambar^i_{j0} & = 0 \\
\Gambar^0_{ij} & = 0 \\
\Gambar^k_{ij} & = \Gamma^k_{ij}
\quad \text{(Christoffel symbols of $\bH^n$).}
\end{aligned}\right.\]
\end{lemma}

The induced metric on $\Sigma$ is given by
\[
g(X, Y) 
= \bbar(\overline{X}, \overline{Y}) 
= b(X, Y) + V^2 \nabla_X f \nabla_Y f.
\]
The second fundamental form $\sffb$ of $\Sigma$ is given by
\[
\begin{split}
\sffb(\overline{X}, \overline{Y})
&= 
\frac{1}{\left|\nablabar F\right|} 
\hessbardd{\overline{X}}{\overline{Y}} F\\
&= 
\frac{1}{\left|\nablabar F\right|}\left[\hessbardd{X}{Y} F
+ \nabla_X f \hessbardd{\partial_0}{Y} F
+ \nabla_Y f \hessbardd{X}{\partial_0} F\right.\\
&\qquad 
\left.+ \nabla_X f \nabla_Y f 
\hessbardd{\partial_0}{\partial_0} F\right] \\
&= 
\frac{1}{\sqrt{V^{-2} + \left|df\right|^2}} \left[\hessdd{X}{Y} f
+ \frac{\nabla_X f \nabla_Y V + \nabla_X V \nabla_Y f}{V}
+ V \<df, dV\> \nabla_X f \nabla_Y f\right].
\end{split}
\]
Using component notation we get
\begin{equation*} %\label{eqSFF}
\sffbdd{i}{j} 
= 
\frac{V}{\sqrt{1 + V^2\left|df\right|^2}} 
\left[\hessdd{i}{j} f + \frac{\nabla_i f \nabla_j V 
+ \nabla_i V \nabla_j f}{V} + V \<df, dV\> \nabla_i f \nabla_j f
\right].
\end{equation*}
The metric $g$ and its inverse are given by
\begin{equation*}
\begin{aligned}
g_{ij} & = b_{ij} + V^2 \nabla_i f \nabla_j f,\\
g^{ij} & = b^{ij} - \frac{V^2 \nabla^i f \nabla^j f}{1 + V^2 |df|^2}.
\end{aligned}
\end{equation*}
We compute the mean curvature of $\Sigma$,
\[ \begin{split}
\Hbar 
&= 
g^{ij} \sffdd{i}{j} \\
&= 
\frac{1}{\left|\nablabar F\right|} 
\left(b^{ij} - \frac{V^2 \nabla^i f \nabla^j f}{1 + V^2 |df|^2} \right) \\ 
&\qquad
\cdot \left[ \hessdd{i}{j} f 
+ \frac{\nabla_i f \nabla_j V + \nabla_i V \nabla_j f}{V} 
+ V \<df, dV\> \nabla_i f \nabla_j f \right] \\
&= 
\frac{1}{\left|\nablabar F\right|} 
\left[\Delta f + 2 \left\<df, \frac{dV}{V}\right\> 
+ V \<df, dV\> |df|^2\right. \\
&\qquad 
-\frac{V^2}{1+V^2 |df|^2}\left(\<\hess f, df\otimes df\>
+ 2 |df|^2 \left\<df, \frac{dV}{V}\right\> \right. \\
&\qquad
\left. \left. 
+ V^2 |df|^4 \left\<df, \frac{dV}{V}\right\>\right)\right] \\
&= 
\frac{1}{\left|\nablabar F\right|} 
\left[
\Delta f - \frac{ V^2 \<\hess f, df\otimes df\>}{1+V^2 |df|^2}
+ \frac{2 + V^2 |df|^2}{1+V^2 |df|^2} \left\<df, \frac{dV}{V}\right\> 
\right],
\end{split} \]
or 
\begin{equation*}
\Hbar = 
\frac{1}{\left|\nablabar F\right|} 
\left[
\Delta f - \frac{ V^2 \<\hess f, df\otimes df\>}{1+V^2 |df|^2}
+ \left(1 + \frac{1}{1+V^2 |df|^2}\right) 
\left\<df, \frac{dV}{V}\right\>
\right].
\end{equation*}
The norm of the second fundamental form of $\Sigma$ is given by
\[ \begin{split}
\left| \sffb\right |^2_g
&= 
g^{ik} g^{jl} \sffbdd{i}{j} \sffbdd{k}{l}\\
&= 
\left(b^{ik} - \frac{V^2 \nabla^i f \nabla^k f}{1 + V^2 |df|^2}\right) 
\left(b^{jl} - \frac{V^2 \nabla^j f \nabla^l f}{1 + V^2 |df|^2}\right)
\sffbdd{i}{j} \sffbdd{k}{l} \\
&= 
b^{ik} b^{jl} \sffbdd{i}{j} \sffbdd{k}{l}
- 2 \frac{V^2 b^{ik} \nabla^j f \nabla^l f}{1 + V^2 |df|^2}
\sffbdd{i}{j} \sffbdd{k}{l}
+ \frac{V^4 \nabla^i f \nabla^j f \nabla^k f \nabla^l f}
{\left(1 + V^2 |df|^2\right)^2} 
\sffbdd{i}{j} \sffbdd{k}{l}\\
&= 
\underbrace{\left|\sffb\right|_b^2}_{(A)}
\underbrace{- 2 \frac{V^2 b^{ik} \nabla^j f \nabla^l f}
{1 + V^2 |df|^2}\sffbdd{i}{j} \sffbdd{k}{l}}_{(B)}
+ 
\underbrace{\left( \frac{V^2 \sffb(\nabla f, \nabla f)}
{1+ V^2 |df|^2} \right)^2}_{(C)}.
\end{split} \]
We compute each term separately. First
\[ \begin{split}
(A)
&= 
\left|\sffb\right|_b^2 \\
&= 
\frac{V^2}{1+V^2 |df|^2} 
\left[\left|\hess f\right|^2 
+ 2 |df|^2 \left|\frac{dV}{V}\right|^2 
+ 2 \left\<df, \frac{dV}{V}\right\>^2 \right. \\
&\qquad 
+ V^4 |df|^4 \left\<df, \frac{dV}{V}\right\>^2 
+ 4 \left\<\hess f, df\otimes \frac{dV}{V}\right\>\\
&\qquad \left. + 2 V^2 \left\<df, \frac{dV}{V}\right\> 
\<\hess f, df\otimes df\> 
+ 4 V^2 |df|^2 \left\<df, \frac{dV}{V}\right\>^2 \right].
\end{split} \]
Next,
\[ \begin{split}
(B)
&= 
- 2 \frac{V^2 b^{ik}}{1 + V^2 |df|^2} 
\nabla^j f \sffbdd{i}{j} \nabla^l f \sffbdd{k}{l}\\
&= 
-2 \frac{V^4}{1 + V^2 |df|^2}\left|\sffb(\nabla f, \cdot)\right|^2 \\
&= 
-2 \frac{V^4}{(1 + V^2 |df|^2)^2}\left|\hess f(\nabla f, \cdot) 
+ (1+V^2 |df|^2) \left\<df, \frac{dV}{V}\right\> df 
+ |df|^2 \frac{dV}{V}\right|^2 \\
&= 
- 2 \frac{V^4}{(1 + V^2 |df|^2)^2} 
\left[\left|\hess f(\nabla f, \cdot)\right|^2 
+ (1+V^2 |df|^2)^2 \left\<df, \frac{dV}{V}\right\>^2 |df|^2 
+ |df|^4 \left|\frac{dV}{V}\right|^2\right.\\
&\qquad 
+ 2 (1+V^2 |df|^2) \hess f(\nabla f, \nabla f) 
\left\<df, \frac{dV}{V}\right\> 
+ 2 |df|^2 \left\<\hess f,\nabla f\otimes\frac{\nabla V}{V}\right\> \\
&\qquad 
+ \left. 2 (1+V^2 |df|^2) |df|^2 
\left\<df, \frac{dV}{V}\right\>^2 \right],
\end{split} \]
and finally
\[ \begin{split}
(C)
&= 
\left( \frac{V^2 \sffb(\nabla f, \nabla f)}{1+ V^2 |df|^2} \right)^2 \\
&= 
\frac{V^6}{\left(1 + V^2 |df|^2\right)^3} 
\left[\nabla^i f \nabla^j f \hessdd{i}{j} f 
+ 2 |df|^2 \left\<df, \frac{dV}{V}\right\> 
+ |df|^4 V^2 \left\<df, \frac{dV}{V}\right\>\right]^2 \\
&=
\frac{V^6}{\left(1 + V^2 |df|^2\right)^3} 
\left[\nabla^i f \nabla^j f \hessdd{i}{j} f 
+ (2 + V^2 |df|^2) |df|^2 \left\<df, \frac{dV}{V}\right\>\right]^2 \\
&= 
\frac{V^2}{1 + V^2 |df|^2}
\left[\frac{V^2 \<\hess f, df \otimes df\>}{1+ V^2 |df|^2} 
+ \left(1 + \frac{1}{1+V^2 |df|^2}\right)
V^2 |df|^2 \left\<df, \frac{dV}{V}\right\>\right]^2 .
\end{split} \]
Hence
\[ \begin{split}
\lefteqn{\Hbar^2 - |\sffb|_g^2}\\
&= 
\frac{V^2}{1+V^2|df|^2}
\left( \left[
\Delta f - \frac{ V^2 \<\hess f, df\otimes df\>}{1+V^2 |df|^2} 
+ \left(1 + \frac{1}{1+V^2 |df|^2}\right) 
\left\<df, \frac{dV}{V}\right\>\right]^2\right.\\
&\qquad 
-\left[\frac{V^2 \<\hess f, df\otimes df\>}{1+ V^2 |df|^2} 
+ \left(1 + \frac{1}{1+V^2 |df|^2}\right) V^2 |df|^2 
\left\<df, \frac{dV}{V}\right\>\right]^2\\
&\qquad 
-\left|\hess f\right|^2 - 2 |df|^2 \left|\frac{dV}{V}\right|^2 
- 2 \left\<df, \frac{dV}{V}\right\>^2 
- V^4 |df|^4 \left\<df, \frac{dV}{V}\right\>^2 \\
&\qquad 
- 4 \left\<\hess f, df\otimes \frac{dV}{V}\right\> 
- 2 V^2 \left\<df, \frac{dV}{V}\right\> \<\hess f, df\otimes df\> 
- 4 V^2 |df|^2 \left\<df, \frac{dV}{V}\right\>^2\\
&\qquad 
+ 2 \frac{V^2}{1 + V^2 |df|^2}
\left[\left| \hess f(\nabla f, \cdot)\right|^2 
+ (1+V^2 |df|^2)^2 \left\<df, \frac{dV}{V}\right\>^2 |df|^2 
+ |df|^4 \left|\frac{dV}{V}\right|^2\right.\\
&\qquad 
+ 2 (1+V^2 |df|^2) \hess f(\nabla f, \nabla f) 
\left\<df, \frac{dV}{V}\right\> 
+ 2 |df|^2 \left\<\hess f,\nabla f\otimes\frac{\nabla V}{V}\right\> \\
&\qquad 
+ \left.\left. 
2 (1+V^2 |df|^2) |df|^2 \left\<df, \frac{dV}{V}\right\>^2 
\right]\right) ,
\end{split}\]
and
\[\begin{split}
\lefteqn{\Hbar^2 - |\sffb|_g^2}\\
&= 
\frac{V^2}{1+V^2|df|^2}
\left(\left[
\Delta f + (2 + V^2 |df|^2) \left\<df,\frac{dV}{V}\right\>
\right]\right.\\
&\qquad 
\cdot \left[ 
\Delta f 
- \frac{2 V^2}{1+V^2|df|^2} \left\<\hess f, df \otimes df\right\> 
+ (1 - V^2|df|^2) \left(1+\frac{1}{1+V^2|df|^2} \right)
\left\<df,\frac{dV}{V}\right\>
\right]\\
&\qquad 
- \left|\hess f\right|^2 
+ 2 \frac{V^2}{1 + V^2 |df|^2} \left|\hess f(\nabla f,\cdot)\right|^2 
- \frac{2}{1+V^2|df|^2} |df|^2 \left|\frac{dV}{V}\right|^2 \\
&\qquad 
+ \left(-2 + 2 V^2 |df|^2 + V^4 |df|^4\right) 
\left\<df, \frac{dV}{V}\right\>^2 
- \frac{4}{1+V^2 |df|^2} 
\left\<\hess f,\nabla f\otimes\frac{\nabla V}{V}\right\>\\
&\qquad 
\left. + 2 V^2 \left\<\hess f, df\otimes df\right\> 
\left\<df,\frac{dV}{V}\right\>\right) \\
&= 
\frac{V^2}{1+V^2|df|^2}
\left[ \left( \Delta f \right)^2 - \left|\hess f\right|^2 
+ 2 \frac{V^2}{1 + V^2 |df|^2} 
\left(\left|\hess f(\nabla f, \cdot)\right|^2 
- \Delta f \left\<\hess f, df\otimes df\right\> \right)\right.\\
&\qquad 
+ \left(2 + \frac{2}{1+V^2|df|^2} \right) 
\Delta f \left\<df,\frac{dV}{V}\right\> 
- \frac{2 V^2}{1+V^2|df|^2} 
\left\<\hess f, df\otimes df\right\> 
\left\<df,\frac{dV}{V}\right\>\\
&\qquad 
+ \left.\frac{2}{1+V^2 |df|^2} \left\<df,\frac{dV}{V}\right\>^2 
- \frac{2}{1+V^2|df|^2} |df|^2 \left|\frac{dV}{V}\right|^2 
- \frac{4}{1+V^2 |df|^2} 
\left\<\hess f, df\otimes \frac{dV}{V}\right\>\right] .
\end{split} \]
By taking the trace of the Gauss equation for $\Sigma$,
we finally arrive at the following formula for the scalar
curvature $\scal$ of $\Sigma$
\begin{equation}\label{eqScalRen}
\begin{split}
\lefteqn{\scal + n(n-1)}\\
&= \Hbar^2 - |\sffb|_g^2\\
&= 
\frac{V^2}{1+V^2|df|^2}\left[ \left(\Delta f\right)^2 
- \left|\hess f\right|^2 
+ 2 \frac{V^2}{1 + V^2 |df|^2} 
\left(\left|\hess f(\nabla f, \cdot)\right|^2 
- \Delta f \left\<\hess f, df\otimes df\right\>\right)\right. \\
&\qquad 
+ \frac{2}{1+V^2|df|^2} \left\<df,\frac{dV}{V}\right\>
\left(\Delta f - V^2 \left\<\hess f, df\otimes df\right\> 
+ \left\<df,\frac{dV}{V}\right\> \right)
+ 2 \left\<df,\frac{dV}{V}\right\> \Delta f \\
&\qquad 
\left. - \frac{2}{1+V^2|df|^2} |df|^2 \left|\frac{dV}{V}\right|^2 
- \frac{4}{1+V^2 |df|^2} 
\left\<\hess f, df\otimes \frac{dV}{V}\right\>\right].
\end{split}
\end{equation}

In view of \cite[Proof of Theorem~5]{LamGraph} and
\cite[Definition~3.3]{HerzlichMassFormulae}, we compute
\[
\divg^b \left[\frac{1}{1+V^2|df|^2} \left(
V \divg^b e - V d \tr^b e - e(\nabla V, \cdot) + (\tr^b e) dV
\right)\right]
\]
with $e = V^2 df \otimes df$. We have
\[ \begin{split}
& V \divg^b e - V d \tr^b e - e(\nabla V, \cdot) + (\tr^b e) dV\\
&\qquad = 
2 V^2 \left\<df, dV\right\> df + V^3 \Delta f df 
+ V^3 \<\hess f, df \otimes \cdot\> \\
& \qquad \qquad 
- V d \tr^b(V^2|df|^2) - V^2 \left\<df, dV\right\> df + V^2 |df|^2 dV \\
&\qquad =
V^3 \Delta f df - V^3 \<\hess f, df \otimes \cdot\> - V^2 |df|^2
dV + V^2 \left\<df, dV\right\> df
\end{split} \]
and
\[ \begin{split}
& \divg^b \left(V \divg^b e - V d \tr^b e - e(\nabla V, \cdot) 
+ (\tr^b e) dV\right)\\
&\qquad =
\divg^b \left(V^3 \Delta f df - V^3 \<\hess f, df \otimes \cdot\> 
- V^2 |df|^2 dV + V^2 \left\<df, dV\right\> df\right) \\
&\qquad =
3 V^2 \Delta f \<df, dV\> + V^3 \<d \Delta f, df\> + V^3 (\Delta f)^2 \\
&\qquad \qquad 
- 3 V^2 \<\hess f, df \otimes dV\> - V^3 \<\divg^b \hess f, df\> 
- V^3 |\hess f|^2\\
&\qquad \qquad 
- 2 V |df|^2 |dV|^2 - 2 V^2 \<\hess f, df \otimes dV\> 
- V^2 |df|^2 \Delta V\\
&\qquad \qquad 
+ 2 V \<df, dV\>^2 + V^2 \<\hess f, dV \otimes df\> 
+ V^2 \<df \otimes df, \hess V \> + V^2 \<df, dV\> \Delta f\\
&\qquad =
V^3 \left[(\Delta f)^2 - |\hess f|^2\right] 
- 4 V^2 \<\hess f, df \otimes dV\> + 4 V^2 \<df, dV\> \Delta f \\
&\qquad \qquad 
+ V^3 \<d \Delta f - \divg^b \hess f, df\> - (n-1) V^3 |df|^2\\
&\qquad \qquad 
+ 2 V \<df, dV\>^2 - 2 V |df|^2 |dV|^2 \\
&\qquad =
V^3 \left[(\Delta f)^2 - |\hess f|^2\right] 
- 4 V^2 \<\hess f, df \otimes dV\> + 4 V^2 \<df, dV\> \Delta f \\
&\qquad \qquad 
+ 2 V \<df, dV\>^2 - 2 V |df|^2 |dV|^2.
\end{split} \]
Further, 
\[ \begin{split}
& \left\<d\left(\frac{1}{1+V^2|df|^2}\right), 
V \divg^b e - V d \tr^b e - e(\nabla V, \cdot) + (\tr^b e) dV\right\> \\
&\qquad = 
\left\<\frac{-2 V |df|^2 dV 
- 2 V^2 \<\hess f, df \otimes \cdot\>}{(1+V^2|df|^2)^2}, 
V^3 \Delta f df - V^3 \<\hess s, df \otimes \cdot\> - V^2 |df|^2 dV 
+ V^2 \left\<df, dV\right\> df\right\>\\
&\qquad = 
\frac{-2}{(1+V^2 |df|^2)^2} \left[V^4 \Delta f \<df, dV\> 
- 2|df|^2 V^4 \<\hess f, df \otimes dV\> 
- V^3 |df|^4 |dV|^2 + V^3 |df|^2 \<df, dV\>^2\right. \\
&\qquad \qquad 
\left. + V^5 \Delta f \<\hess f, df \otimes df\> - V^5
|\<\hess f, df \otimes\cdot\>|^2 
+ V^4 \<df, dV\> \<\hess f, df \otimes df\>\right], 
\end{split} \]
so
\[ \begin{split}
&\divg^b \left[\frac{1}{1+V^2|df|^2} 
\left(V \divg^b e - V d \tr^b e - e(\nabla V, \cdot) 
+ (\tr^b e) dV\right)\right]\\
&\qquad =
\frac{1}{1+V^2 |df|^2} \left[V^3 
\left((\Delta f)^2 - |\hess f|^2\right) 
- 4 V^2 \<\hess f, df \otimes dV\> 
+ 4 V^2 \<df, dV\> \Delta f\right.\\
&\qquad \qquad \qquad 
\left. + 2 V \<df, dV\>^2 - 2 V |df|^2 |dV|^2\right]\\
&\qquad \qquad 
- 
\frac{2}{(1+V^2 |df|^2)^2} \left[V^4 \Delta f \<df, dV\> 
- |df|^2 V^4 \<\hess f, df \otimes dV\> - V^3 |df|^4 |dV|^2 
+ V^3 |df|^2 \<df, dV\>^2\right.\\
&\qquad \qquad \qquad
\left. + V^5 \Delta f \<\hess f, df \otimes df\> 
- V^5 |\<\hess f, df \otimes\cdot\>|^2 
+ V^4 \<df, dV\> \<\hess f, df \otimes df\>\right]\\
&\qquad =
\frac{1}{1+V^2 |df|^2} 
\left[V^3 \left((\Delta f)^2 - |\hess f|^2\right) 
- \frac{2}{1+V^2 |df|^2} \left(V^5 \Delta f 
\<\hess f, df \otimes df\> 
- V^5 |\<\hess f, df \otimes\cdot\>|^2\right)\right.\\
& \qquad \qquad 
- \frac{4 V^2}{1+V^2 |df|^2} \<\hess f, df \otimes dV\>
+ \frac{2 V}{1+V^2 |df|^2} \left(\<df, dV\>^2 - |df|^2 |dV|^2\right)\\
& \qquad \qquad \qquad 
\left. - \frac{2 V^4}{1+V^2 |df|^2} \<df, dV\> 
\<\hess f, df \otimes df\> 
+ \left(2 + \frac{1}{1+V^2 |df|^2}\right) 
\Delta f |df|^2 \<df, dV\>\right].
\end{split} \]

Comparing this formula with Equation \eqref{eqScalRen} we get
\begin{equation} \label{eqScalCurvDivg}
\begin{split}
&V \left(\scal + n(n-1)\right) \\
&\qquad = 
\divg^b \left[\frac{1}{1+V^2|df|^2} \left(
V \divg^b e - V d \tr^b e - e(\nabla V, \cdot) + (\tr^b e) dV
\right)\right],
\end{split}
\end{equation}
where $e = V^2 df \otimes df$.

%%%%%%%%%%%%%%%%%%%%%%%%%%%%%%%%%%%%%%%%%%%%%%%%%%%%%%%%%%%%%%%%%%%%%%%%%
\subsection{A mass formula}
%%%%%%%%%%%%%%%%%%%%%%%%%%%%%%%%%%%%%%%%%%%%%%%%%%%%%%%%%%%%%%%%%%%%%%%%%

We now integrate Formula \eqref{eqScalCurvDivg} from the previous section 
over an outer domain under the additional condition that $f$ is locally 
constant on the boundary.

\begin{lemma} \label{lmIntegrCourbScal}
Let $\Omega \subset \bH^n$ be a relatively compact open subset of
$\bH^n$ with smooth boundary. Let $f : \bH^n \setminus \Omega \to \bR$
be an asymptotically hyperbolic function which is locally constant on 
$\partial \Omega$ and such that $df \neq 0$ at every point of 
$\partial \Omega$. Then
\begin{equation} \label{eqIntegrCourbScal}
H_\Phi(V) 
= 
\int_{\bH^n \setminus \Omega} 
\frac{V [\scal + n(n-1)]}{\sqrt{1 + V^2 |df|^2}} \, d\mu^g 
+ 
\int_{\partial \Omega} H V \frac{V^2 |df|^2}{1+V^2 |df|^2} \, d\mu^b.
\end{equation}
Here $H$ is the mean curvature of $\partial \Omega$ with respect to 
the metric $b$.
\end{lemma}

\begin{proof}
Let $\nu$ denote the outgoing unit normal to $\partial \Omega$ and let
$\nu_r = \partial_r$ be the normal to the spheres of constant $r$. 
From Formula \eqref{eqScalCurvDivg} we get 
\[\begin{split}
&\int_{\bH^n \setminus \Omega} V \left(\scal + n(n-1)\right) \, d\mu^b \\
&\qquad
= \lim_{r \to \infty} \int_{B_r(0) \setminus \Omega} 
V \left(\scal + n(n-1)\right) \, d\mu^b \\
&\qquad 
= \lim_{r \to \infty} \int_{S_r(0)} \frac{1}{1+V^2|df|^2} 
\left(V \divg^b e - V d \tr^b e 
- e(\nabla V, \cdot) + (\tr^b e) dV\right)(\nu_r) \, d\mu^b\\
&\qquad \qquad
- \sum_i \int_{\partial \Omega} \frac{1}{1+V^2|df|^2} 
\left(V \divg^b e - V d \tr^b e - e(\nabla V, \cdot) 
+ (\tr^b e) dV\right)(\nu) \, d\mu^b\\
&\qquad 
= H_\Phi(V) - \sum_i \int_{\partial \Omega} \frac{1}{1+V^2|df|^2} 
\left(V \divg^b e - V d \tr^b e - e(\nabla V, \cdot) 
+ (\tr^b e) dV\right)(\nu) \, d\mu^b.
\end{split}\]
Here we used that $e = V^2 df \otimes df$ satisfies \eqref{decay3} 
to replace the factor $\frac{1}{1+V^2|df|^2}$ by $1$ in the limit of 
the outer boundary integral. We next compute the integral over 
$\partial \Omega$. We will do the calculations assuming that 
$\nu = \frac{\nabla f}{|\nabla f|}$, the case 
$\nu = - \frac{\nabla f}{|\nabla f|}$ is similar. 
The last two terms are
\[ 
- e(\nabla V, \nu) + (\tr^b e) dV(\nu)
= - V^2 \<df, dV\> \<df, \nu\> + V^2 |df|^2 \<dV, \nu\>
= 0,
\]
and the first two give
\[ \begin{split}
V \divg^b e(\nu) - V d \tr^b e (\nu)
&= 
2 V^2 \<df, dV\> df(\nu) + V^3 (\Delta f) df(\nu) 
+ V^3 \hess f(\nabla f, \nu) \\
&\qquad 
- 2 V^2 |df|^2 dV(\nu) - 2 V^3 \hess f (\nabla f, \nu) \\
&= 
V^3 (\Delta f) df(\nu) - V^3 \hess f(\nabla f, \nu).
\end{split} \]
We next use the following formula for the Laplacian of $f$ on
$\partial \Omega$, 
\[
\Delta f 
= \Delta^{\partial \Omega} f + \hess f(\nu, \nu) + H df(\nu).
\]
Since $f$ is locally constant on $\partial \Omega$ we obtain 
\[
V \divg^b e(\nu) - V d \tr^b e (\nu) 
= 
V^3 H df(\nu)^2 
= 
V^3 H |df|^2.
\] 
Hence,
\[
\int_{\bH^n \setminus \Omega} V \left(\scal + n(n-1)\right) \, d\mu^b 
= 
H_\Phi(V) - \sum_i \int_{\partial \Omega} 
V H \frac{V^2 |df|^2}{1+V^2|df|^2} \, d\mu^b.
\]
It then suffices to note that $d\mu^g = \sqrt{1 + V^2 |df|^2} d\mu^b$ 
to prove Formula \eqref{eqIntegrCourbScal}.
\end{proof}

%%%%%%%%%%%%%%%%%%%%%%%%%%%%%%%%%%%%%%%%%%%%%%%%%%%%%%%%%%%%%%%%%%%%%%%%%
\section{Penrose type inequalities}
\label{secPenrose}
%%%%%%%%%%%%%%%%%%%%%%%%%%%%%%%%%%%%%%%%%%%%%%%%%%%%%%%%%%%%%%%%%%%%%%%%%

%%%%%%%%%%%%%%%%%%%%%%%%%%%%%%%%%%%%%%%%%%%%%%%%%%%%%%%%%%%%%%%%%%%%%%%%%
\subsection{Horizon boundary}
%%%%%%%%%%%%%%%%%%%%%%%%%%%%%%%%%%%%%%%%%%%%%%%%%%%%%%%%%%%%%%%%%%%%%%%%%

From now on we assume that $|df| \to \infty$ at $\partial \Omega$, it 
then follows that the boundary is a minimal hypersurface, or a horizon. 
This can be seen by taking the double over the boundary of the graph of 
$f$. The double is then a $C^1$ Riemannian manifold for which the 
original boundary is the fixed point set of a reflection, and thus the 
boundary is minimal. It is not hard to prove that there can be no other 
minimal surface in the graph which encloses $\partial \Omega$.

From Lemma~\ref{lmIntegrCourbScal} we conclude the following 
proposition.

\begin{proposition}
\label{prop_penrose_graph}
Let $\Omega \subset \bH^n$ be a relatively compact open subset of
$\bH^n$ with smooth boundary. Let $f : \bH^n \setminus \Omega \to \bR$
be an asymptotically hyperbolic function which is locally constant on 
$\partial \Omega$ and such that $|df| \to \infty$ at $\partial \Omega$. 
Further assume that $\scal \geq -n(n-1)$. Then 
\begin{equation} \label{penrose_general_boundary}
H_\Phi(V) 
\geq
\int_{\partial \Omega} V H \, d\mu^b.
\end{equation}
\end{proposition}

Applying the Hoffman-Spruck inequality or the Minkowski formula we get 
estimates of the boundary term in \eqref{penrose_general_boundary} and 
conclude the following Theorem.

\begin{theorem} 
\label{THM-AH-Penrose-Graph-HS}
Let $\Omega \subset \bH^n$ be a relatively compact open subset of
$\bH^n$ with smooth boundary. Assume that $\Omega$ contains an inner ball 
centered at the origin of radius $r_0$. 
Let $f : \bH^n \setminus \Omega \to \bR$ be an asymptotically 
hyperbolic function which is locally constant on $\partial \Omega$ and 
such that $|df| \to \infty$ at $\partial \Omega$. Further assume that 
$\scal \geq -n(n-1)$ and that $\partial \Omega$ has non-negative mean 
curvature $H \geq 0$. Then 
\begin{equation} \label{AH-Penrose-HS}
H_\Phi(V) \geq 
\frac{n-2}{2^{n-1} n^{\frac{n}{n-1}}} V(r_0) \omega_{n-1}
\left( \frac{|\partial \Omega|}{\omega_{n-1}} \right)^{\frac{n-2}{n-1}}
\end{equation}
and
\begin{equation} \label{AH-Penrose-M}
H_\Phi(V) \geq (n-1) V(r_0) |\partial \Omega|.
\end{equation}
\end{theorem}

\begin{proof}
The Hoffman-Spruck inequality,
\cite{hoffman_spruck_74, otsuki_75, tanno_83}, applied to a compact
hypersurface $M$ of hyperbolic space $\bH^n$ tells us that 
\begin{equation} \label{hoffman-spruck}
\left( \int_M h^{\frac{n-1}{n-2}} \, d\mu^b \right)^{\frac{n-2}{n-1}}
\leq 
C_n \int_M \left( |d h| + h |H| \right) \, d\mu^b
\end{equation}
for any smooth non-negative function $h$ on $M$. Here
\[
C_n =
2^{n-1} \frac{n}{n-2} \left(\frac{n}{\omega_{n-1}} \right)^{\frac{1}{n-1}}. 
\]
Setting $h \equiv 1$ and $M=\partial \Omega$ in \eqref{hoffman-spruck} 
yields \eqref{AH-Penrose-HS}.

The estimate \eqref{AH-Penrose-M} follows from the Minkowski formula in 
hyperbolic space, see \cite[Equation~(4')]{MontielRos} with the point 
$a = (1,0,\dots,0)$ (note that in the cited article the mean curvature 
is defined as an average and not a sum). 
\end{proof}

Neither of the inequalities \eqref{AH-Penrose-HS} and \eqref{AH-Penrose-M}
is optimal, so we do not get a characterization of the case of equality 
in the corresponding Penrose type inequalities.

%%%%%%%%%%%%%%%%%%%%%%%%%%%%%%%%%%%%%%%%%%%%%%%%%%%%%%%%%%%%%%%%%%%%%%%%%
\subsection{Changing to the Euclidean metric}
\label{changetoeucl}
%%%%%%%%%%%%%%%%%%%%%%%%%%%%%%%%%%%%%%%%%%%%%%%%%%%%%%%%%%%%%%%%%%%%%%%%%

We will now find an estimate of the boundary term in 
\eqref{penrose_general_boundary} by changing to the Euclidean metric 
$\btil \definedas b + dV \otimes dV$. In the hyperboloid model of hyperbolic 
space this transformation can be viewed as the vertical projection of 
$\bH^n$ onto $\bR^n \subset \bR^{n,1}$.

\begin{lemma}
Let $\nu$ be the outgoing unit normal to $\partial \Omega$. The second 
fundamental form of $\partial \Omega$ with respect to the
metric $\btil$ is given by
\begin{equation*}
\sffrdd{i}{j} 
= \frac{V}{\sqrt{V^2 - \<dV, \nu\>^2}}
\left(
\sffdd{i}{j} 
- \frac{\nabla^k V \nabla_k \psi}{V} b_{ij}
\right),
\end{equation*}
where $\psi$ is a defining function for $\partial \Omega$ such that $\nabla \psi = \nu$.
Further, we have
\begin{equation}\label{eqIntHV}
\begin{split}
\int_{\partial \Omega} HV \, d\mu
&= 
\int_{\partial \Omega} \Htil \, d\widetilde{\mu} 
+ (n-1) \int_{\partial \Omega} \< dV, \nu\> \, d\mu \\
&\qquad 
+ \int_{\partial \Omega} 
\frac{1}{1+|\nabla^T V|^2} 
\left(\sff(\nabla^T V, \nabla^T V)V
- |\nabla^T V|^2 \< dV, \nu\> \right) \, d\mu,
\end{split}
\end{equation}
where 
$\nabla^T V$ is the gradient of $V$ for the metric induced by $b$
on $\partial \Omega$.
\end{lemma}

\begin{proof}
The second fundamental form of $\partial \Omega$
with respect to the metric $\btil$ is given by
\[
\sffrdd{i}{j} = \frac{1}{|d\psi|_{\btil}} \hesstildd{i}{j} \psi.
\] 
We compute 
\[
\hesstildd{i}{j} \psi - \hessdd{i}{j} \psi 
= 
\left(\Gamma^k_{ij} - \Gamtil^k_{ij}\right) \partial_k \psi.
\]
At the center point $p_0$ of normal coordinates for the metric $b$ 
the difference between the two Christoffel symbols is given by 
\[\begin{split}
\Gamtil^k_{ij} - \Gamma^k_{ij}
&= 
\frac{1}{2} \btil^{kl} \left(\nabla_i \btil_{lj} 
+ \nabla_j \btil_{il} - \nabla_l \btil_{ij}\right)\\
&= 
\frac{1}{2} 
\left(b^{kl} - \frac{\nabla^k V \nabla^l V}{1+|dV|^2} \right) 
\left[\nabla_i (\nabla_l V\nabla_j V) + \nabla_j (\nabla_i V \nabla_l V) 
- \nabla_l (\nabla_i V\nabla_j V)\right]\\
&= 
\left(b^{kl} - \frac{\nabla^k V \nabla^l V}{1+|dV|^2} \right) 
\nabla_l V \hessdd{i}{j} V \\
&= 
\frac{\nabla^k V}{1+|dV|^2} \hessdd{i}{j} V \\
&= 
\frac{\nabla^k V}{V} b_{ij} ,
\end{split}\]
where we used that $\hess^b V = V b$ and $1 + |dV|^2 = V^2$ in 
the last line. Further, we have 
\[
|d\psi|_{\btil} 
= \sqrt{1 - \frac{\<dV, \nu\>^2}{1+|dV|^2}}
= \frac{1}{V} \sqrt{V^2 - \<dV, \nu\>^2}.
\] 
Hence,
\[\begin{split}
\sffrdd{i}{j} 
&= \frac{V}{\sqrt{V^2 - \<dV, \nu\>^2}}
\left(
\hessdd{i}{j} \psi 
- \frac{\nabla^k V \nabla_k \psi}{V} b_{ij}
\right)\\
&= \frac{V}{\sqrt{V^2 - \<dV, \nu\>^2}}
\left(
\sffdd{i}{j} 
- \frac{\nabla^k V \nabla_k \psi}{V} b_{ij}
\right).
\end{split}\]
We take the trace of this formula with respect to the metric $\btil$. 
For this we select an orthogonal basis 
$(e_1, \dots, e_{n-1})$ of $T_{p_0}\partial \Omega$ for the metric $b$ such
that $e_k \in \ker dV$ for $k \geq 2$. An orthogonal basis for the
metric $\btil$ is then given by 
\[\begin{aligned}
\widetilde{e}_1 &= \frac{1}{\sqrt{1 + (\nabla_{e_1} V)^2}} e_1 , \\
\widetilde{e}_k &= e_k \qquad \text{for } k \geq 2.
\end{aligned}\]
Thus, we find 
\[\begin{split}
\Htil
&= 
\sum_{k=1}^{n-1} \sffr(\widetilde{e}_k, \widetilde{e}_k)\\
&= 
\sum_{k=1}^{n-1} \sffr(e_k, e_k) 
- \left(1 - \frac{1}{1 + (\nabla_{e_1} V)^2}\right) \sffr(e_1, e_1)\\
&= 
\frac{V}{\sqrt{V^2 - \<dV, \nu\>^2}}
\left[
\left(H - (n-1)\frac{\< dV, d\psi\>}{V}\right)
- \frac{(\nabla_{e_1} V)^2}{1 + (\nabla_{e_1} V)^2} 
\left(\sff(e_1, e_1) - \frac{\< dV, d\psi\>}{V}\right)
\right] \\
&= 
\frac{1}{\sqrt{V^2 - \<dV, \nu\>^2}}
\left[
\left(HV - (n-1) \< dV, d\psi\> \right)
- \frac{(\nabla_{e_1} V)^2}{1 + (\nabla_{e_1} V)^2} 
\left(\sff(e_1, e_1)V - \< dV, d\psi\> \right)
\right].
\end{split}\]
Next we note that $(\nabla_{e_1} V)^2 = |dV|^2 - \<dV, \nu\>^2$ is 
the norm of $dV$ restricted to the tangent space of $\partial \Omega$. 
Hence the measure $d\widetilde{\mu}$ induced on $\partial \Omega$ by 
$\btil$ is given by 
\[
d\widetilde{\mu} 
= \sqrt{1 + |dV|^2 - \<dV, \nu\>^2} \, d\mu
= \sqrt{V^2 - \<dV, \nu\>^2} \, d\mu
\]
where $d\mu$ is the measure induced on $\partial \Omega$ by $b$.
Finally we conclude
\[\begin{split}
\int_{\partial \Omega} \Htil \, d\widetilde{\mu}
&= 
\int_{\partial \Omega} 
\left( HV - (n-1) \< dV, d\psi\> \right)
\, d\mu \\
&\qquad 
- \int_{\partial \Omega} 
\frac{|\nabla_{e_1} V|^2}{1+|\nabla_{e_1} V|^2} 
\left(\sff(e_1, e_1)V - \< dV, d\psi\> \right)
\, d\mu. 
\end{split}\]
\end{proof}

The assumption $\sffr > 0$ is equivalent to
\[
\sff > \frac{\nabla^k V \nabla_k \psi}{V} b,
\]
where this inequality is to be understood as an inequality between
quadratic forms on $T\partial\Omega$. This notion of convexity is
not invariant under the action of isometries of the hyperbolic space.
Since $|dV| < V$, it is natural to replace this assumption by 
\[
\sff \geq b.
\]
This new assumption is equivalent to the definition of h-convexity (see
for example \cite{BorisenkoMiquel}). Assuming that $\Omega$ is h-convex,
we get the following inequality from \eqref{eqIntHV}.

\begin{equation}\label{eqIntHV2}
\int_{\partial \Omega} HV \, d\mu
\geq 
\int_{\partial \Omega} \Htil \, d\widetilde{\mu} 
+ (n-1) \int_{\partial \Omega} \< dV, \nu\> \, d\mu.
\end{equation}

We estimate the first term of the right-hand side by the 
Aleksandrov-Fenchel inequality, see \cite[Theorem 2]{GuanLi},
\cite[Lemma 12]{LamGraph}, \cite{SchneiderBook} or \cite{BuragoZalgaller}.
\[
\int_{\partial \Omega} \Htil \, d\widetilde{\mu}
\geq 
(n-1) \omega_{n-1} 
\left(\frac{|\partial\Omega|_{\btil}}{\omega_{n-1}}\right)^{\frac{n-2}{n-1}}
\geq 
(n-1) \omega_{n-1} 
\left(\frac{|\partial\Omega|_b}{\omega_{n-1}}\right)^{\frac{n-2}{n-1}}.
\]
Equality in the first inequality here implies that $\partial \Omega$
is a round sphere in the Euclidean metric $\btil$, equality in the second
inequality tells us that it must be centered at the origin.

To estimate the second term of \eqref{eqIntHV2}, we rely on
\cite[Theorem 2]{BorisenkoMiquel}. Assuming that the origin is the center
of an inner ball of $\Omega$ and denoting by $l$ the distance from the 
origin, we have, for any point $p \in \partial\Omega$,
\[
\<\nu, \nabla l\>
\geq
\frac{\tanh^2 \frac{l}{2}(p) + \tau}{\tanh \frac{l}{2}(p)(1 + \tau)},
\]
where $\tau = \tanh \frac{r_0}{2}$ and $r_0$ is the radius of an inner
ball of $\Omega$. Hence, setting $t = \tanh \frac{l}{2}(p)$, we have
\[ \begin{split}
\int_{\partial \Omega} \< dV, \nu\> \, d\mu
&=
\int_{\partial \Omega} \sinh l \< \nabla l, \nu\> \, d\mu\\
&\geq
\int_{\partial \Omega} \sinh l \frac{t^2 + \tau}{t (1 + \tau)} \, d\mu\\
&=
\int_{\partial \Omega} \frac{2t}{1-t^2} \frac{t^2+\tau}{t(1+\tau)} \, d\mu\\
&=
\frac{2}{1+\tau} \int_{\partial \Omega} \frac{t^2+\tau}{1-t^2} \, d\mu\\
&\geq 
\frac{2}{1+\tau} \frac{\tau^2+\tau}{1-\tau^2} |\partial \Omega|_b\\
&\geq 
\sinh r_0 |\partial \Omega|_b.
\end{split} \]
It is also easy to check that the equality
\[
\int_{\partial \Omega} \< dV, \nu\> \, d\mu=\sinh r_0 |\partial \Omega|_b
\]
holds if and only if $\Omega$ is the ball of radius $r_0$ centered at the origin.

Combining the last two estimates, we get the following inequality:
\begin{equation}\label{eqIntHV3}
\int_{\partial \Omega} HV \, d\mu
\geq 
(n-1) \omega_{n-1} \left[
\left(\frac{|\partial\Omega|_b}{\omega_{n-1}}\right)^{\frac{n-2}{n-1}}
+ \sinh r_0 \frac{|\partial\Omega|_b}{\omega_{n-1}}
\right].
\end{equation}

From Proposition \ref{prop_penrose_graph} and Inequality \eqref{eqIntHV3},
we immediately get the following theorem.

\begin{theorem} \label{THM-AH-Penrose-Graph}
Let $\Omega$ be a non-empty h-convex subset of $\bH^n$ admitting an 
inner ball centered at the origin of radius $r_0$. Let 
$f: \bH^n\setminus\Omega \to \bR$ be an asymptotically hyperbolic 
function such that $f$ is locally constant on $\partial \Omega$,
$|df| \to \infty$ at $\partial \Omega$. Assume that the scalar 
curvature $\scal$ of its graph is larger than $-n(n-1)$. Then
\begin{equation} \label{AH-Penrose-Graph}
H_\Phi(V) \geq 
(n-1) \omega_{n-1} 
\left[
\left( \frac{|\partial \Omega|}{\omega_{n-1}} \right)^{\frac{n-2}{n-1}} 
+ \sinh r_0 \frac{|\partial \Omega|}{\omega_{n-1}}\right].
\end{equation}
Moreover, equality holds in \eqref{AH-Penrose-Graph} if and only if $\scal = -n(n-1)$ and 
$\partial \Omega$ is round sphere centered at the origin.
\end{theorem}

We make a couple of remarks concerning this theorem.

\begin{remark}~
\begin{enumerate}
\item 
If $\Omega$ is a ball of radius $r$ then $r_0 = r$ and
\[|\partial\Omega| = \omega_{n-1} \sinh^{n-1} r_0,\] so 
\eqref{AH-Penrose-Graph} coincides with the standard Penrose 
inequality \eqref{AH-Penrose} in this case.
\item 
The second term of \eqref{eqIntHV2} can be written as follows,
\[
\int_{\partial\Omega} \<dV, \nu\>\, d\mu
= \int_{\Omega} \Delta V\, d\mu
= n \int_{\Omega} V\, d\mu.
\]
Thus this term may be thought of as a volume integral. Compare with
\cite{Schwartz}. Let $V_p \definedas \cosh d_b(p, \cdot)$. Changing
the origin $p$ of hyperbolic space leads to considering the function 
\[
p \mapsto \int_\Omega V_p \, d\mu.
\] 
It is fairly straightforward to see that this function is proper and 
strictly convex. So there exists a unique point $p_0$ such that, 
choosing $p_0$ as the origin, this integral is minimal. Obviously, 
$p_0 \in \Omega$. From symmetry considerations this point can be seen 
to coincide with the center of an inner ball for many $\Omega$'s. 
\item 
It follows from the previous remark, that it is possible to prove a 
Penrose inequality when $\Omega$ has several (h-convex) components 
assuming for example that if one component contains the origin then 
it is the center of one of its inner balls. For each of the other 
components, simply remark that translating them using an isometry of 
the hyperbolic space so that the origin becomes the center of one of 
its inner balls makes the integral $\int H V \, d\mu$ smaller. Hence 
we get the following inequality.
\[
H_\Phi(V) \geq 
(n-1) \omega_{n-1} 
\sum_i \left[
\left( \frac{|\partial \Omega_i|}{\omega_{n-1}} \right)^{\frac{n-2}{n-1}} 
+ \sinh r_i \frac{|\partial \Omega_i|}{\omega_{n-1}}\right],
\]
where $\Omega_i$ are the connected components of 
$\Omega$ and $r_i$ is the inner radius of $\Omega_i$.
\end{enumerate}
\end{remark}

%%%%%%%%%%%%%%%%%%%%%%%%%%%%%%%%%%%%%%%%%%%%%%%%%%%%%%%%%%%%%%%%%%%%%%%%%
\section{Rigidity}
\label{secRigidity}
%%%%%%%%%%%%%%%%%%%%%%%%%%%%%%%%%%%%%%%%%%%%%%%%%%%%%%%%%%%%%%%%%%%%%%%%%

In this section we will prove the rigidity statement concluding Theorem 
\ref{THM-Main}. The scheme of the proof we give differs very little from 
\cite{HuangWu2}. As a first step, we prove the following proposition which 
is similar to \cite[Theorem 3]{HuangWu2}.

\begin{proposition}\label{propMeanCurv}
Let $f : \bH^n \setminus \overline{\Omega}\rightarrow \bR$ be a function 
satisfying the assumptions of Theorem \ref{THM-Main} and let $\Sigma$ be
its graph. Assume further that $\Omega$ is convex. Then the mean curvature
$\Hbar$ of $\Sigma$ does not change sign.
\end{proposition}

The proof of this proposition requires several preliminary results. The 
main observation is the fact that the assumption $\scal \geq -n(n-1)$ is 
equivalent to $\left|\sffb\right|^2 \leq \Hbar^2$. This follows at once 
from the Gauss equation. In particular, any point $p \in \Sigma$ such 
that $\Hbar(p) = 0$ has $\sffb(p) = 0$. We denote by $\Sigma_0$ the set 
of such points,
\[
\Sigma_0 \definedas
\{p \in \operatorname{int}(\Sigma) \mid \Hbar(p) = 0\},
\]
where $\operatorname{int}(\Sigma) = 
\Sigma \setminus (\partial \Omega \times \bR)$.

\begin{lemma} \label{lmStructM0}
Let $\Sigma'_0$ be a connected component of $\Sigma_0$. Then $\Sigma'_0$
lies in a codimension 1 hyperbolic subspace tangent to $\Sigma$ at every
point of $\Sigma'_0$.
\end{lemma}

\begin{proof}
Let $V_{(0)}, \ldots, V_{(n)}$ be as in Section \ref{secAH} and let $\nu$ be
the unit normal vector field of $\Sigma$ in $\bH^{n+1}$. For any vector
$X \in T\Sigma$ at a point of $\Sigma'_0$ we have 
\[\begin{aligned}
\nablabar_X (d V_{(i)}(\nu))
&= \hessbardd{X}{\nu} V_{(i)} + dV_{(i)} (\nablabar_X \nu)\\
&= V_{(i)} \bbar(X, \nu) + dV_{(i)} (S(X))
&= 0,
\end{aligned}\]
where $S(X)$ denotes the Weingarten operator which is zero by assumption.
From \cite[Theorem 4.4]{Morse} we conclude that $dV_{(i)}(\nu)$ is
constant on $\Sigma'_0$. If we consider $\bH^{n+1}$ as the unit hyperboloid 
in Minkowski space $\bR^{n+1,1}$, then the $V_{(i)}$ are the coordinate 
functions of $\bR^{n+1,1}$ restricted to $\bH^{n+1}$ so $\nu$ is a constant 
vector in $\bR^{n+1,1}$. Further, $\nu$ is tangent to $\bH^{n+1}$ so it is 
orthogonal to the position vector in $\bR^{n+1,1}$. This means that $\nu$ 
is everywhere orthogonal to a linear subspace $W \subset \bR^{n+1,1}$. We 
conclude that $\Sigma'_0 \subset W \cap \bH^{n+1} \simeq \bH^{n}$. 
\end{proof}

The next result is taken from \cite[Proposition 2.1]{HuangWu1}.

\begin{lemma}[A matrix inequality]\label{lmMatrix}
 Let $A = (a_{ij})$ be a symmetric $n \times n$ matrix. Set
\[\begin{aligned}
\sigma_1(A) &\definedas \sum_{i=1}^n a_{ii}, \\
\sigma_1(A|k) &\definedas \left(\sum_{i=1}^n a_{ii}\right)-a_{kk},\\
\sigma_2(A) &\definedas \sum_{1 \leq i < j \leq n} 
\left(a_{ii} a_{jj} - a_{ij}^2\right).
\end{aligned}\]
Then we have 
\[\begin{aligned}
\sigma_1(A)\sigma_1(A|k) 
&= 
\sigma_2(A)+\frac{n}{2(n-1)}\sigma_1(A|k)^2\\
&\qquad 
+ \sum_{1 \leq i < j \leq n} a_{ij}^2 + \frac{1}{2(n-1)} 
\sum_{\substack{1 \leq i < j \leq n\\ i\neq k, j\neq k}} (a_{ii} - a_{jj})^2
\end{aligned}\]
for each $1\leq k \leq n$. In particular,
\[
\sigma_1(A)\sigma_1(A|k) \geq \sigma_2(A)+\frac{n}{2(n-1)}\sigma_1(A|k)^2,
\]
where equality holds if and only if $A$ is diagonal and all $a_{ii}$ are
equal for $i = 1, \ldots, n$, $i\neq k$.
\end{lemma}

\begin{proposition}\label{propLevelSet}
Let $\Sigma$ and $s_0$ be given. Assume that $s_0$ is a regular value
for $f$ on $\Sigma$. Set $\Sigma(s_0) = \Sigma \cap f^{-1}(s_0)$. Let
$\nu$ be the unit normal vector field of $\Sigma$ in $\bH^{n+1}$, let $\eta$
be the unit normal vector field to $\Sigma(s_0)$ in $\bH^n \times \{s_0\}$
and let $H(s_0)$ be the mean curvature of $\Sigma(s_0)$ in
$\bH^n \times \{s_0\}$ computed with respect to $\eta$. Then
\[
\<\nu, \eta\> \Hbar H(s_0) 
\geq 
\frac{\scal+n(n-1)}{2} + \frac{n}{2(n-1)}\<\nu, \eta\>^2 H(s_0)^2.
\]
Equality holds at a point in $\Sigma(s_0)$ if and only if
\begin{itemize}
\item $\Sigma(s_0) \in \bH^n \times \{s_0\}$ is umbilic with principal
curvature $\kappa$, and 
\item $\<\nu, \eta\> \kappa$ is a principal curvature of $\Sigma$ with
multiplicity at least $(n-1)$.
\end{itemize}
\end{proposition}

\begin{proof}
Let $p$ be a point in $\Sigma(s_0)$. We compute the second
fundamental form of $\Sigma(s_0)$ in $\bH^{n+1}$ at $p$ in two different
ways. Let $e_1 \in T_p\Sigma$ be a unit vector field
orthogonal to $T_p \Sigma(s_0)$. We denote by $\sffb_0$ the second
fundamental form of $\Sigma(s_0)$ in $\bH^{n+1}$. This is a symmetric
bilinear form on $T_p \Sigma(s_0)$ taking values in the normal bundle
$N_p \Sigma(s_0) \subset T_p \bH^{n+1}$. Further, we denote by $\sff_1$ 
the second fundamental form of $\Sigma(s_0)$ in $\Sigma$ computed with
respect to the vector $e_1$. Since $\bH^n \times \{s_0\}$ is totally
geodesic in $\bH^{n+1}$, we have 
\[
\sffb_0 = S_0 \eta.
\]
Similarly, 
\[
\sffb_0 = \sffb \nu + S_1 e_1.
\] 
Hence, taking the scalar product of the last two equalities with 
$\nu$, we get
\[
\<\eta, \nu\> S_0 = \sffb. 
\]

Let $\{e_2, \ldots, e_n\}$ be an orthonormal basis of $T \Sigma(s_0)$, 
then $\{e_1, \ldots, e_n\}$ is an orthonormal basis of $T_p \Sigma$. 
Set 
\[
\sffb_{ij} \definedas \sffb(e_i, e_j).
\]
Then, using the notation of Lemma \ref{lmMatrix}, we have
\[\begin{aligned}
\sigma_1(\sffb) 
&= \Hbar,\\
\sigma_1(\sffb|1) 
&= \sum_{i=2}^n \sffb(e_i, e_i)\\
&= \<\eta, \nu\>\sum_{i=2}^n S_0(e_i, e_i)\\
&= \<\eta, \nu\> H(s_0),\\
\sigma_2(\sffb) 
&= \frac{1}{2} \left(\Hbar^2 - \left|\sffb\right|^2\right)\\
&= \frac{\scal+n(n-1)}{2}.
\end{aligned}\]
Proposition \ref{propLevelSet} now follows from Lemma \ref{lmMatrix}.
\end{proof}

The proof of Proposition \ref{propMeanCurv} will also require following 
two lemmas, analogous to \cite[Lemma 3.3 and Lemma 3.4]{HuangWu2}.

\begin{lemma}\label{lmLocLevelSets}
Let $W$ be an open subset of $\bH^n$, possibly unbounded. 
Let $p \in \partial W$, and let $B(p)$ be a geodesic open ball in 
$\bH^n$ centered at $p$. Consider 
$f \in C^2(W \cap B(p))\cap C^1(\overline{W} \cap B(p))$ and let 
$\Hbar$ denote the mean curvature of its graph. If $f=C$ and $|df|=0$ on 
$\partial W\cap B(p)$, where $C$ is a constant, and $\Hbar \geq 0$ on 
$W\cap B(p)$ then either $f\equiv C$ in $W\cap B(p)$, or 
\[
\{x \in W \cap B(p) \mid f(x)>C\} \neq \emptyset.
\]
\end{lemma}

\begin{proof}
If $f\equiv C$ then there is nothing to prove. Suppose therefore that 
$f\not\equiv C$ and assume to get a contradiction that $f(x)\leq C$ 
everywhere in $W\cap B(p)$. 

We first note that in fact $f<C$ everywhere in $W\cap B(p)$. Indeed, let 
$q \in W\cap B(p)$ be such that $f(q)=C$. Then $q$ is an interior maximum 
point of $f$ in $W\cap B(p)$, whereas 
\begin{equation*}
\begin{split}
\Hbar 
&= 
\frac{V}{1+V^2\left|df\right|^2} 
\left(b^{ij} - \frac{V^2 \nabla^i f \nabla^j f}{1 + V^2 |df|^2} \right) \\
&\qquad 
\cdot \left[ 
\hessdd{i}{j} f + \frac{\nabla_i f \nabla_j V 
+ \nabla_i V \nabla_j f}{V} + V \<df, dV\> \nabla_i f \nabla_j f 
\right] \geq 0
\end{split}
\end{equation*}
in $W\cap B(p)$, see Section \ref{subsection_comp_scal}. By the Hopf 
strong maximum principle it follows that $f\equiv C$ in $W\cap B(p)$, 
which is a contradiction.

Now suppose that $B(p)=B_r(p)$ is the ball of radius $r$ around $p$. 
Fix a point $q \in B_{r/2}(p)$ and define 
$r' \definedas \sup \{r \mid B_r(q)\subset W\}$. It is clear that 
$B_{r'}(q) \subset W \cap B(p)$ and $\overline{B}_{r'}(q)\cap \partial W
\neq \emptyset$. Consequently, there is a point $s\in\partial W$ such that
the interior sphere condition holds at $s$. Then by the Hopf boundary lemma
\cite[Lemma 3.4]{GilbargTrudinger}, we have $|df|>0$ at $s$, which is a
contradiction. We conclude that $f>C$ holds somewhere in $W\cap B(p)$.
\end{proof}

\begin{definition}\label{defConvexPoint}
Let $W$ be a bounded subset of $\bH^n$ and let $\overline{W}$ be its 
closure. A point $p\in \partial W$ is called convex if there is a 
geodesic $(n-1)$-sphere $S$ in $\bH^n$ passing through $p$ such that 
$\overline{W} \setminus \{p\}$ is contained in the open geodesic ball 
enclosed by $S$.
\end{definition}

Note that every bounded set in $\bH^n \setminus \Omega$ has at least one
convex point. This follows from the assumption that $\Omega$ is convex.
We only sketch the proof of this fact leaving the details to the reader.
Choose a point $p \in W$ and let $q$ be the projection of $p$ onto
$\partial \Omega$. Then the hyperbolic subspace passing through $q$ and
orthogonal to the geodesic joining $p$ to $q$ cuts $\bH^n$ in two
half-spaces, a ``left'' one containing $\Omega$ and a ``right'' one
containing $p$. Then if $O'$ is located very far on the left side of
the geodesic $(qp)$, it is clear that the smallest sphere $S$ centered at
$O'$ containing $\Omega \cup W$ has a non-trivial intersection with
$\partial W$. Any point in $S \cap \partial W$ is then a convex point.

\begin{lemma}\label{lmLocMeanCurv}
Let $W$ be an open bounded subset of $\bH^n$ and let $p \in \partial W$ 
be a convex point. Suppose that 
$f \in C^n(W \cap B(p)) \cap C^1(\overline{W} \cap B(p))$ is such that 
$f=C$ and $|df|=0$ on $\partial W \cap B(p)$ for some constant $C$. If 
the graph of $f$ has scalar curvature $\scal \geq -n(n-1)$, then its mean
curvature $\Hbar$ must change sign in $W\cap B(p)$, unless $f\equiv C$ 
in $W\cap B(p)$.
\end{lemma}

\begin{proof}
Suppose on the contrary that $\Hbar$ does not change sign and
$f \not\equiv 0$. By possibly reversing sign and adding a constant to $f$
we may assume that $\Hbar \geq 0$ and that $C=0$. 

Let $S_r$ be a geodesic $(n-1)$-sphere of radius $r$ as in Definition 
\ref{defConvexPoint}, centered at a point $O'\in\bH^n$, and such that 
$S_r\cap \overline{W}=\{p\}$. Let $\mu$ be a positive number strictly 
less than the distance from $W\setminus B(p)$ to $S_r$. Then for every 
sphere $S_{r'}$ of radius $r'\in (r-\mu, r)$ and centered at $O'$ we 
obviously have $S_{r'}\cap W \subset B(p)$. Let $f_0$ be a continuous 
function on $B(p)$ such that $f_0=f$ on $W\cap B(p)$ and $f_0=0$ on 
$B(p)\setminus W$. Define the function 
\[
g(r') \definedas \sup_{q\in S_{r'}\cap B(p)} f_0 (q)
\]
for $r'\in [r-\mu,r]$. It is easy to check that $g$ is continuous 
and satisfies $g(r)=0$. Next, we observe that by Lemma \ref{lmLocLevelSets} 
the ball $B_{\mu}(p)$ contains a point $q$ such that 
$f_0(q) = \epsilon > 0$. By the Morse-Sard theorem 
\cite[Theorem 7.2]{sard} we may assume that each connected component of 
the corresponding level set 
\[
\Sigma (\epsilon)=\{x \in W\cap B(p) \mid f_0(x)=\epsilon\}
\] 
of $f_0$ inside $W\cap B(p)$ is a $C^n$ hypersurface. It is clear that 
$g([r-\mu,r]) = [0,\epsilon']$, where $\epsilon \leq \epsilon'$, 
and hence 
\[
r_0 \definedas \max_{r'\in[r-\mu, r]} \{r' \mid g(r') = \epsilon\}
\] 
is well-defined. Then $S_{r_0}\cap\Sigma_\epsilon \neq \emptyset$, whereas 
$S_{r'}\cap \Sigma_\epsilon =\emptyset$ for $r_0< r'\leq r$, thus $S_{r_0}$ 
is tangent to $\Sigma (\epsilon)$ at some interior point $q$. Let $U$ 
be the open subset of $W\cap B(p)$ bounded by $S_{r_0}$ and $\partial W$,
\[
U = \{x \in W \cap B(p) | d(O', x) > r_0\},
\]
then $q\in \partial U$. We have $f(q)=\epsilon>f(x)$ for any $x\in U$, 
$\Hbar\geq 0$ holds in $U$, and the interior sphere condition is obviously 
satisfied at $q\in S_{r_0}$. Since $\eta=-\frac{\nabla f}{|df|}$ is orthogonal 
to $\partial U$ at $q$, it is easy to conclude by the Hopf boundary lemma 
that $\eta$ is the inward pointing normal to $\partial U$. Hence $\eta$ 
is the outward pointing normal for both $S_{r_0}$ and $\Sigma (\epsilon)$ 
at $q$. By the comparison principle, the mean curvature $H(\epsilon)$ of
$\Sigma (\epsilon)$ satisfies $H(\epsilon)>0$ at $q$. On the other 
hand, since the scalar curvature of the graph of $f$ is nonnegative, by 
Proposition \ref{propLevelSet} at $q$ we have 
\[
\<\nu, \eta\> \Hbar H(\epsilon) \geq 0.
\]
Here $\<\nu, \eta\> < 0$ since 
$\nu = \frac{(\nabla f, -V^{-2})}{\sqrt{V^{-2} + |df|^2}}$, $\Hbar\geq 0$, and 
if $\Hbar=0$ then $H(\epsilon) = 0$. This means that 
$H(\epsilon) \leq 0$ at $q$, which is a contradiction. Hence $\Hbar$ must 
change sign in $W\cap B(p)$.
\end{proof}

\begin{proof}[Proof of Proposition \ref{propMeanCurv}]
We assume by contradiction that $\Hbar$ changes sign, both sets
$\{\Hbar>0\}$ and $\{\Hbar<0\}$ are nonempty in $\Sigma$. Our first
observation is that each connected component of these two sets is
unbounded. Indeed, let $\Sigma_+$ be a bounded connected component of
$\{\Hbar>0\}$ and let $\partial_0 \Sigma_+$ be its outer boundary
component. By Lemma \ref{lmStructM0} we know that $\partial_0 \Sigma_+$
lies in an $n$-dimensional hyperbolic subspace $\Pi$. We view $\bH^{n+1}$
as $\Pi \times \bR$ with the metric 
$b+V^2 d\widetilde{s} \otimes d\widetilde{s}$, and we let $W$ be a subset 
of $\{\widetilde{s}=0\}$ bounded by $\partial_0 \Sigma_+$. Then in some
neighborhood of $\partial W$ we can write $\Sigma_+$ as the graph of a 
function $u$ such that $u=0$ and $|du|=0$ on $\partial W$. Now, considering 
a sufficiently small ball $B(p)$ around $p\in \partial W$, we immediately 
arrive at the contradiction, since $\Hbar$ must change sign in 
$W\cap B(p)$ by Lemma \ref{lmLocMeanCurv}.

We have just seen that if $\Sigma_+$ is a connected component of 
$\{\Hbar>0\}$ then it must be unbounded, and the same is clearly true for 
a connected component $\Sigma_-$ of $\{\Hbar<0\}$. Moreover, it follows 
by Proposition \ref{prop_unbounded_boundary} in Appendix \ref{appopensets} 
that one of the connected components of its boundary $\partial \Sigma_+$ 
is unbounded, and the same holds for $\partial \Sigma_-$. Let us denote 
such an unbounded component by $\partial_0 \Sigma_+$. By 
Lemma \ref{lmStructM0} we know that $\partial_0 \Sigma_+$ lies in an 
$n$-dimensional hyperbolic subspace $\Pi$ tangent to $\Sigma$ at every 
point of $\partial_0 \Sigma_+$. Since $\Sigma$ is asymptotically 
hyperbolic, $f$ tends to a constant value $C$ at infinity, so the fact 
that $\partial_0\Sigma_+$ is unbounded forces $\Pi$ to coincide with 
the plane $\{s=C\}$.

The component $\Sigma_+$ is the graph of $f$ over some open subset $W$ 
of $\bH^n$. Moreover, there is an unbounded component $\partial_0 W$ of 
the boundary $\partial W$ such that $f=C$ and $|df|=0$ on $\partial_0 W$. 
By Lemma \ref{lmLocLevelSets} there exists $q\in W$ such that 
$f(q)=C+\epsilon$ for some $\epsilon>0$. By the Morse-Sard theorem we 
know that there is an $\epsilon$ such that $C+\epsilon$ is a regular 
value of $f$, so that the corresponding level set 
$f^{-1}(C+\epsilon) = \{p \mid f(p) = C+\epsilon\}$ is a 
$C^n$ hypersurface with $|df|>0$ at each point. Suppose that $U$ is a 
connected component of $\{H\geq 0\}$ in $\bH^n$ which contains $W$. 
Then, using Proposition \ref{prop_unbounded_boundary} and the fact that 
$f$ tends to $C$ at infinity, it is easy to check that if some connected 
component of $f^{-1}(C+\epsilon)$ intersects $U$, then it is 
contained in $U$. It is also obvious that 
$f^{-1}(C+\epsilon)\cap U$ is nonempty and bounded, so we can 
find a point $p$ in this set which is at the largest distance $d$ from 
the origin $O$ of $\bH^n$. Let $\Sigma(C+\epsilon)$ be the connected 
component of $f^{-1}(C+\epsilon)$ which contains $p$. Then the 
geodesic sphere of radius $d$ centered at $O$ touches 
$\Sigma(C+\epsilon)$ at $p$, and there are no points $x$ such that 
$f(x) \geq C+\epsilon$ in $\{r>d\} \cap U$. Arguing as in the proof 
of Lemma \ref{lmLocMeanCurv}, we can show that 
$\eta\definedas-\frac{\nabla f}{|df|}=\partial_r$ at $p$, that is, $\nu$
is an outgoing normal to $\Sigma(C+\epsilon)$. The mean curvature
$H(C+\epsilon)$ is then positive at $p$, whereas Proposition
\ref{propLevelSet} tells us that $H(C+\epsilon)\leq 0$ at $p$, which is a
contradiction.
\end{proof}

Let $f$ be as in Theorem \ref{THM-Main}. We recall the expressions for 
$g$, $\sffb$, $\Hbar$, and $\scal$ obtained in Section \ref{secAHgraphs}, 
and rewrite them as functions of the arguments $Df$ and $D^2 f$, where 
$Df$ and $D^2 f$ denote the Euclidean gradient and the Euclidean Hessian 
respectively:
\[\begin{aligned}
g^{ij} (Df) 
&= 
b^{ij}-\frac{V^2 f^i f^j}{1+V^2 |df|^2},\\
\sffbdd{i}{j} (Df, D^2 f)
&= 
\frac{V}{\sqrt{1 + V^2 |df|^2}} 
\left[f_{ij}-\Gamma_{ij}^l f_l + \frac{f_i V_j + V_i f_j}{V} 
+ V \langle df, dV \rangle f_i f_j\right],\\
\sffb_j ^i (Df, D^2f) 
&= 
\frac{V}{\sqrt{1 + V^2 |df|^2}} 
\left(b^{ik}-\frac{V^2 f^i f^k}{1+V^2 |df|^2}\right) \\ 
&\qquad
\left( f_{kj}-\Gamma_{kj}^l f_l + \frac{f_k V_j + V_k f_j}{V} 
+ V \langle df, dV \rangle f_k f_j \right),\\
\Hbar (Df, D^2 f) 
&= 
\frac{V}{\sqrt{1 + V^2 |df|^2}} 
\left(b^{ij}-\frac{V^2 f^i f^j}{1+V^2|df|^2}\right)\\ 
&\qquad 
\left( f_{ij}-\Gamma_{ij}^l f_l + \frac{f_i V_j + V_i f_j}{V} 
+ V \langle df, dV \rangle f_i f_j \right),\\
\scal(Df, D^2 f) 
&= -n(n-1) + \Hbar^2 (Df, D^2 f)- \sffb_i^j (Df,D^2 f) \sffb_j^i(Df,D^2 f).
\end{aligned}\]

Following \cite[Section 4]{HuangWu2}, we will now prove maximum principles
for the scalar curvature equation $\scal(Df, D^2 f) + n(n-1)=0$. 
The lemma below concerns ellipticity of this equation.

\begin{lemma}\label{lmLinScalCurv}
\[
\frac{\partial \scal}{\partial f_{ij}} 
= \frac{2V}{\sqrt{1+V^2 |df|^2}} \left(\Hbar g^{ij} - \sffb^{ij} \right).
\]
\end{lemma} 
\begin{proof}
A straightforward computation gives 
\[\begin{aligned}
\frac{\partial \scal}{\partial f_{ij}} 
&= 
2\Hbar \frac{\partial \Hbar}{\partial f_{ij}} 
-2\sffb^k_l \frac{\partial \sffb^l_k}{\partial f_{ij}}\\ 
&=
\frac{2V}{\sqrt{1+V^2 |df|^2}} \left(\Hbar g^{ij}- \sffb^k_l g^{lm} 
\frac{\partial f_{mk}}{\partial f_{ij}} \right)\\ 
&=
\frac{2V}{\sqrt{1+V^2 |df|^2}}\left(\Hbar g^{ij}-\sffb^{ij} \right).
\end{aligned}\]
\end{proof}

\begin{proposition}
Let $f$ be as in Theorem \ref{THM-Main}. Suppose that the scalar 
curvature $\scal$ and the mean curvature $\Hbar$ of its graph satisfy 
$\scal \geq -n(n-1)$ and $\Hbar \geq 0$ Then the matrix 
$\left(\Hbar g^{ij}-\sffb^{ij}\right)$ is positive semi-definite everywhere 
in $\bH^n\setminus \overline{\Omega}$.
\end{proposition}
\begin{proof}
We work at a point $p\in \bH^n\setminus \overline{\Omega}$. Since 
$\Hbar g^{ij} - \sffb^{ij} = 
\sum _k \left(\Hbar \delta_k^j - \sffb^j _k \right)g^{ik}$, where $g^{ik}$ 
is positive definite, we only need to show that 
$\left(\Hbar \delta_k^j - \sffb^j _k \right)$ is positive semi-definite. 
After possibly rotating the coordinates, we may assume that 
$\sffb = \left(\sffb^j _k\right)=\diag (\lambda_1, \ldots, \lambda_n)$. 
Then, in the notation of Lemma \ref{lmMatrix}, we have
\[
\left(\Hbar \delta_k^j - \sffb^j _k \right)
= \diag \left(\sigma_1 (\sffb|1), \ldots, \sigma_1 (\sffb|n)\right).
\]
By Lemma \ref{lmMatrix} it follows that
\[
\sigma_1(\sffb) \sigma_1 (\sffb|k)
\geq 
\sigma_2(\sffb)+\frac{n}{2(n-1)}\left(\sigma_1 (\sffb|k)\right)^2,
\]
for $k=1,\ldots,n$. If $\sigma_1 (\sffb)=\Hbar > 0$, since 
$\sigma_2 (\sffb)=\frac{1}{2}(\scal+n(n-1))\geq 0$, it is obvious that 
$\sigma_1 (\sffb|k)\geq 0$ for every $k=1,\ldots,n$. Otherwise if
$\Hbar = 0$ then $\sffb = 0$ and hence $\sigma_1 (\sffb|k) = 0$.
This proves that $\sigma_1(\sffb|k) \geq 0$.
\end{proof}

In the next two propositions we prove versions of the maximum principle 
for the scalar curvature equation, the first one for points in the interior 
and the second one for points on the boundary.

\begin{proposition}
\label{PropIntMaxPrinciple}
Let $f_i : \bH^n \setminus \overline{\Omega}\rightarrow \bR$, $i=1,2$, 
be two functions satisfying the assumptions of Theorem \ref{THM-Main}. 
Suppose that $f_1\geq f_2$ in $\bH^n \setminus \overline{\Omega}$, and that
$f_i$, $i=1,2$, satisfy the inequalities
\[\begin{aligned}
\scal(Df_1, D^2 f_1) = -n(n-1), &\qquad \Hbar(Df_1, D^2 f_1)\geq 0,\\
\scal(Df_2, D^2 f_2) \geq -n(n-1), &\qquad \Hbar(Df_2, D^2 f_2)\geq 0
\end{aligned}\]
in $\bH^n \setminus \overline{\Omega}$. If the matrix 
$\left(\Hbar g^{ij}- \sffb^{ij}\right)$ is positive definite in
$\bH^n\setminus \overline{\Omega}$ for either $f_1$ or $f_2$, and if
$f_1=f_2$ at some point of $\bH^n \setminus \overline{\Omega}$, then
$f_1\equiv f_2$ in $\bH^n \setminus \overline{\Omega}$.
\end{proposition}

\begin{proof}
We consider the scalar curvature operator as 
$\scal(p,\xi) \in C^1\left(\bR^n,\bR^n\times\bR^n\right)$. Then 
\[\begin{aligned}
0
&\geq 
\scal (Df_1, D^2 f_1)-\scal (Df_2, D^2 f_2) \\
&=
\scal (D f_1, D^2 f_1)-\scal (D f_1, D^2 f_2) + 
\scal (D f_1, D^2 f_2)-\scal (D f_2, D^2 f_2) \\
&=
\sum_{i,j} a^{ij} ((f_1)_{ij}-(f_2)_{ij}) + \sum_i b^i ((f_1)_i-(f_2)_i),
\end{aligned}\] 
where
\[
b^i = 
\int_0^1 \frac{\partial \scal}{\partial p_i}(t Df_1+(1-t)Df_2, D^2 f_2) \, dt,
\]
and
\[
a^{ij}
=
\int_0^1 \frac{\partial \scal}{\partial \xi_{ij}}
(Df_1, tD^2 f_1+(1-t) D^2 f_2)\, dt.
\]
Note that by Lemma \ref{lmLinScalCurv} we have
\[\begin{aligned}
a^{ij}
&=
\int_0^1 \frac{\partial \scal}{\partial \xi_{ij}}
(Df_1, tD^2 f_1+(1-t) D^2 f_2)\, dt \\
&= 
\frac{2V}{\sqrt{1+V^2 |df|^2}} \int_0^1
\left(\Hbar g^{ij}-\sffb^j_k g^{ik}\right)(Df_1, tD^2 f_1+(1-t)D^2 f_2)\, dt \\
&= 
\frac{2V}{\sqrt{1+V^2 |df|^2}} \left[ \int_0 ^1 t 
\left(\Hbar (Df_1, D^2 f_1)g^{ij}(Df_1)
- \sffb^j_k (Df_1, D^2 f_1) g^{ik}(Df_1)\right)\, dt \right. \\
&\qquad+
\left. \int_0 ^1 (1-t) \left(\Hbar (Df_1, D^2 f_2)g^{ij}(Df_1)
-\sffb^j_k (Df_1, D^2 f_2) g^{ik}(Df_1)\right) \, dt \right] \\
&= 
\frac{V}{\sqrt{1+V^2 |df|^2}} \left[ 
\left(\Hbar (Df_1, D^2 f_1)g^{ij}(Df_1)
-\sffb^j_k (Df_1, D^2 f_1) g^{ik}(Df_1)\right)\right. \\
&\qquad+ 
\left.\left(\Hbar (Df_1, D^2 f_2)g^{ij}(Df_1)
-\sffb^j_k (Df_1, D^2 f_2) g^{ik}(Df_1)\right) \right].
\end{aligned}\]
If $f_1=f_2$ at $p\in \bH^n \setminus \overline{\Omega}$, then $p$ is a 
local minimum point of $f_1-f_2$, hence $Df_1=Df_2$ at $p$. Consequently, 
$a^{ij}$ is positive definite at $p$. By continuity, $a^{ij}$ is positive 
definite in some open neighborhood $U$ of $p$ in 
$\bH^n \setminus \overline{\Omega}$. Then $f_1 \equiv f_2$ in $U$ by the 
Hopf strong maximum principle. It follows that the set 
$\{p\in \bH^n \setminus \overline{\Omega} \mid f_1 (p)=f_2(p)\}$ is both 
open and closed in $\bH^n \setminus \overline{\Omega}$. Since 
$\bH^n \setminus \overline{\Omega}$ is connected, we conclude that 
$f_1\equiv f_2$ everywhere $\bH^n \setminus \overline{\Omega}$.
\end{proof}

\begin{proposition}
\label{PropBoundMaxPrinciple} 
Let $f_i : \bH^n \setminus \overline{\Omega}\rightarrow \bR$, $i=1,2$, be 
functions satisfying the assumptions of Theorem \ref{THM-Main}. Suppose 
that $f_1\geq f_2\geq C$ in $\bH^n \setminus \overline{\Omega}$, and that 
$f_i$, $i=1,2$, satisfy the inequalities
\[\begin{aligned}
\scal(Df_1, D^2 f_1) = -n(n-1), &\qquad \Hbar(Df_1, D^2 f_1)\geq 0,\\
\scal(Df_2, D^2 f_2) \geq -n(n-1), &\qquad \Hbar(Df_2, D^2 f_2)\geq 0
\end{aligned}\]
in $\bH^n \setminus \overline{\Omega}$. If the matrix 
$\left(\Hbar g^{ij}- \sffb^{ij}\right)$ is positive definite in $\bH^n\setminus \Omega$ for either 
$f_1$ or $f_2$, and if $f_1=f_2=C$ on $\partial\Omega$, then 
$f_1\equiv f_2$ in $\bH^n \setminus \overline{\Omega}$.
\end{proposition}

\begin{proof}
Let $\Sigma_i$ denote the graph of $f_i$, $i=1,2$. Take 
$p \in \partial \Sigma_1=\partial \Sigma_2\subset \{s=C\}$, and let 
$\nu(p)$ be the common normal to $\Sigma_i$, $i=1,2$, 
at this boundary point. Suppose that $\Pi$ is the hyperbolic subspace 
orthogonal to $\nu(p)$, then $\Pi$ is isometric to $\bH^n$. Let $B_r(p)$ 
be a geodesic ball of radius $r$ in $\Pi$ centered at $p$, and let 
$U = B_r(p)\cap \{s> C\}$. If $r$ is sufficiently small, we can write 
$\Sigma_i$ near $p$ as the graph of $\widetilde{f}_i: U \rightarrow \bR$, 
$i=1,2$, in $U\times \bR$ with the metric 
$b+V^2 d\widetilde{s} \otimes d\widetilde{s}$, where $b$ is the hyperbolic 
metric on $U$, and $\widetilde{s}$ is the coordinate along the $\bR$-factor. 
It is obvious that $\nabla\widetilde{f}_i=0$ at $p$ for $i=1,2$. We also 
have $\widetilde{f}_1 \geq \widetilde{f}_2$ in $U$, and
\[\begin{aligned}
\scal(D\widetilde{f}_1, D^2 \widetilde{f}_1)=-n(n-1), 
&\qquad 
\Hbar(D\widetilde{f}_1, D^2 \widetilde{f}_1)\geq 0, \\
\scal(D\widetilde{f}_2, D^2 \widetilde{f}_2)\geq -n(n-1), 
&\qquad 
\Hbar(D\widetilde{f}_2, D^2 \widetilde{f}_2)\geq 0.
\end{aligned}\]
Moreover, either $\widetilde{f}_1$ or $\widetilde{f}_2$ has positive 
definite matrix $\left(\Hbar g^{ij}- \sffb^{ij}\right)$ at $p$. Arguing as in the 
proof of Proposition \ref{PropIntMaxPrinciple}, we see that 
$(\widetilde{f}_1-\widetilde{f}_2)$ satisfies
\[
0 \geq
\sum_{i,j} a^{ij} ((\widetilde{f}_1)_{ij}-(\widetilde{f}_2)_{ij}) 
+ \sum_i b^i ((\widetilde{f}_1)_i-(\widetilde{f}_2)_i),
\]
where we may assume (after decreasing $r$) that $a^{ij}$ is positive 
definite on $U$. If we assume that $\widetilde{f}_1>\widetilde{f}_2$ in 
$U$ then by the Hopf boundary lemma we have 
$\nabla (\widetilde{f}_1-\widetilde{f}_2)(p)\neq 0$, a contradiction. 
Consequently, $\widetilde{f}_1(q)=\widetilde{f}_2(q)$ at some interior 
point $q\in \bH^n \setminus \overline{\Omega}$. Application of 
Proposition \ref{PropIntMaxPrinciple} completes the proof.
\end{proof}

We recall that $\rho \definedas \sinh(r)$. The hyperbolic metric $b$
takes the form 
\[
b = \frac{(d\rho)^2}{1+\rho^2} + \rho^2 \sigma,
\]
and the function $V = \cosh(r) = \sqrt{1+\rho^2}$.

\begin{proposition}
The second fundamental form of the graph given by \eqref{eqHeightAdSSchw}
is given by
\[
\sffb
= 
-\frac{n-2}{2} \frac{\sqrt{2m} \rho^{-\frac{n}{2}}}
{1+\rho^2 - \frac{2m}{\rho^{n-2}}} d\rho^2
+ \sqrt{2m} \rho^{-\frac{n}{2}+2} \sigma.
\]
In particular, the principal curvatures of the graph $\Sigma$ are
$-\frac{n-2}{2} \sqrt{2m} \rho^{-\frac{n}{2}}$ with multiplicity $1$ and
$\sqrt{2m} \rho^{-\frac{n}{2}}$ with multiplicity $n-1$. The mean curvature
$\Hbar$ is given by
\[
\Hbar = \frac{n}{2} \sqrt{2m} \rho^{-\frac{n}{2}}.
\]
In particular, the quadratic form
\[
\Hbar g - \sffb 
= 
(n-1) \frac{\sqrt{2m} \rho^{-\frac{n}{2}}}{1+\rho^2 - \frac{2m}{\rho^{n-2}}} 
d\rho^2 + \frac{n-2}{2} \sqrt{2m} \rho^{-\frac{n}{2}+2} \sigma
\]
is positive definite.
\end{proposition}

\begin{proof}
Straightforward calculations.
\end{proof}

We are now ready to prove the result on rigidity for the case of equality 
in the last inequality of Theorem \ref{THM-Main}. From Theorem 
\ref{THM-AH-Penrose-Graph} we know that in this case $\scal = -n(n-1)$ 
and $\partial \Omega \subset \bH^n$ is a round sphere centered at the 
origin. The result thus follows from the next theorem.

\begin{thm} \label{thm_rigidity}
Let $f : \bH^n \setminus \Omega \to \bR$ be an asymptotically hyperbolic 
function which satisfies the assumptions of Theorem \ref{THM-Main}
and such that the graph of $f$ has constant scalar curvature $\scal = -n(n-1)$. 
Also assume that $\partial \Omega$ is a round sphere centered 
at the origin and that $df(\eta)(x) \to \infty$ as $x \to \partial \Omega$ 
where $\eta$ is the outward normal of the level sets of $f$. 
Then the graph of $f$ is isometric to the $t=0$ slice of the anti-de 
Sitter Schwarzschild space-time, as described in Section \ref{secAdS}.
\end{thm}

\begin{proof}
By adding a constant to $f$ we assume that $f=0$ on $\partial \Omega$.
From Proposition \ref{propMeanCurv} we know that $\Hbar$ does not change
sign. Proposition \ref{propLevelSet} together with the fact that $H$ is
positive on $\partial \Omega$ tells us that $\Hbar \geq 0$ on the boundary, 
and thus $\Hbar \geq 0$ everywhere. The maximum principle applied to 
$\Hbar$ together with $df(\eta) \to +\infty$ at $\partial \Omega$ tells
us that $\limsup_{x \to \infty} f(x) > 0$. Since $f$ is an asymptotically 
hyperbolic function we conclude that $\lim_{x \to \infty} f(x) = C$ where 
$0 < C < \infty$.

Let $f_{\text{\rm AdS-Schw}}$ be the asymptotically hyperbolic function 
whose graph is isometric to the $t=0$ slice of anti-de Sitter Schwarzschild 
space-time, with mass parameter $m$ such that its horizon is exactly 
the sphere $\partial \Omega$. This function vanishes on $\partial \Omega$ 
and has $\lim_{x \to \infty} f_{\text{\rm AdS-Schw}} = C_0$ where $0 < C_0 < \infty$.
 
If $C \leq C_0$ we set $u_{\lambda} = f_{\text{\rm AdS-Schw}} + \lambda$ for 
$\lambda \geq 0$. If $\lambda$ is large enough then $u_{\lambda} > f$. We 
decrease $\lambda$ until finally $u_{\lambda}(p)=f(p)$ at a point $p$,
possibly $p=\infty$. If $p$ is an interior point then Proposition
\ref{PropIntMaxPrinciple} tells us that $u_{\lambda} \equiv f$, if $p$ is a
boundary point then Proposition \ref{PropBoundMaxPrinciple} tells us that
$u_{\lambda} \equiv f$. There is however one more situation to consider,
namely when $u_\lambda>f$ and $\lim_{x\rightarrow\infty} (u_\lambda-f)=0$.
Since both the graph of $u_\lambda$ and the graph of $f$ have
$\scal=-n(n-1)$, arguing as in the proof of Proposition
\ref{PropIntMaxPrinciple} we conclude that $u_\lambda-f$ satisfies the
equation
\[
\sum_{i,j} a^{ij} (u_\lambda-f)_{ij} + \sum_i b^i (u_\lambda-f)_i=0.
\]
In this case, the Hopf strong maximum principle tells us that $u_\lambda-f$
attains its positive maximum  either at an interior point or at a point of
$\partial \Omega$. Let us denote this point by $q$ and suppose that
$(u_\lambda-f)(q)=\beta>0$. Clearly, $f \geq u_\lambda-\beta$, and
$f(q)=(u_\lambda-\beta)(q)$. By either Proposition \ref{PropIntMaxPrinciple}
or Proposition \ref{PropBoundMaxPrinciple} we conclude that
$u_\lambda-\beta\equiv f$. 

If $C > C_0$ we set $v_{\lambda} = f_{\text{\rm AdS-Schw}} - \lambda$ for 
$\lambda \geq 0$. For $\lambda$ large enough we have $v_{\lambda} < f$ and
we decrease $\lambda$ until $v_{\lambda}$ hits $f$. Arguing as above it is easy to show that $v_{\lambda} \equiv f$. 

In any case we have found that $f$ and $f_{\text{\rm AdS-Schw}}$ differ by 
a constant, which is the conclusion of the theorem.
\end{proof}

%%%%%%%%%%%%%%%%%%%%%%%%%%%%%%%%%%%%%%%%%%%%%%%%%%%%%%%%%%%%%%%%%%%%%%%%%
\appendix
\section{A property of unbounded open subsets of $\bR^n$} 
\label{appopensets}
%%%%%%%%%%%%%%%%%%%%%%%%%%%%%%%%%%%%%%%%%%%%%%%%%%%%%%%%%%%%%%%%%%%%%%%%%

In this appendix we will prove the following result on the boundary 
components of an unbounded open subset of $\bR^n$.

\begin{proposition} \label{prop_unbounded_boundary}
Let $H: \bR^n \to \bR$, $n \geq 2$, be a continuous function which takes 
both positive and negative values. Assume that each connected component 
of $H^{-1}( (0,\infty) )$ and $H^{-1}( (-\infty,0) )$ is unbounded. Then 
there is a connected component of $H^{-1}(0)$ which is unbounded.
\end{proposition}

To prove the proposition we use the following lemma.

\begin{lemma}
Let $K \subset \bR^n$, $n \geq 2$, be compact and connected. Let $U$ be 
the unbounded connected component of $\bR^n \setminus K$. Then 
$U_{\epsilon} \definedas \{ x \in U \mid d(x,K) < \epsilon \}$ is connected. 
\end{lemma}

\begin{proof}
Let $F \definedas \bR^n \setminus U$. This set is closed and bounded and 
therefore compact. We show that $F$ is connected. Let $f: F \to \{0,1\}$ 
be continuous. Then $f$ is constant on $K$. Take $x \in F \setminus K$. 
For $0 \neq a \in \bR^n$ consider the half-line $\{ x + ta \mid 0 \leq t \}$. 
Let $t_0$ be the smallest number so that $x + t_0 a \in K$. Then the line 
segment $\{ x + ta \mid 0 \leq t \leq t_0 \}$ is a subset of $F$, and we 
conclude that $f$ must be constant on $F$ so $F$ is connected. Next 
define $F_{\epsilon} \definedas \{ x \in \bR^n \mid d(x,F) < \epsilon \}$. 
Since $F_{\epsilon} = \cup_{p \in F} B_{\epsilon}(p)$ this is a connected set. 
Note that $F_{\epsilon} = U_{\epsilon} \cup F$. The Mayer-Vietoris sequence 
for homology tells us that
\[
\dots \to H_1 (\bR^n) \to H_0 (U_{\epsilon}) 
\to H_0 (U) \oplus H_0 (F_{\epsilon})
\to H_0 (\bR^n) \to 0,
\]
from which we conclude that $U_{\epsilon}$ is connected.
\end{proof}

\begin{proof}[Proof of Proposition \ref{prop_unbounded_boundary}]
Let $V$ be a connected component of $H^{-1}((0,\infty))$. Let 
$V' \subset \bR^n$ be the image of $V$ when compactifying $\bR^n$ with 
a point at infinity and then removing a point $p$ lying in an unbounded 
component of $\bR^n \setminus V$. The set $V'$ is open, bounded and 
connected, so the closure $K \definedas \overline{V'}$ is compact and 
connected. Let $\partial^{\infty} K$ be the part of the boundary of $K$ 
facing the unbounded component of $\bR^n \setminus K$. Since the 
intersection of a nested sequence of compact connected sets is connected 
we conclude from the Lemma that $\partial^{\infty} K$ is connected. Going 
back to $V$ this means that the union $\partial^{\infty} V \cup \{ \infty \}$ 
is connected, where $\partial^{\infty} V$ is the part of the boundary 
facing the component of $\bR^n \setminus V$ containing $p$. From this we 
see that all components of $\partial^{\infty} V$ must be unbounded, since 
if there was a bounded component this would remain disconnected from the 
others when adding the point at infinity. Finally, every component of 
$\partial^{\infty} V$ is contained in some connected component of 
$H^{-1}(0)$, and those components of $H^{-1}(0)$ are therefore unbounded.
\end{proof}

%%%%%%%%%%%%%%%%%%%%%%%%%%%%%%%%%%%%%%%%%%%%%%%%%%%%%%%%%%%%%%%%%%%%%%%%%
\providecommand{\bysame}{\leavevmode\hbox to3em{\hrulefill}\thinspace}
\providecommand{\MR}{\relax\ifhmode\unskip\space\fi MR }
\providecommand{\MRhref}[2]{%
  \href{http://www.ams.org/mathscinet-getitem?mr=#1}{#2}
}
\providecommand{\href}[2]{#2}

%%%%%%%%%%%%%%%%%%%%%%%%%%%%%%%%%%%%%%%%%%%%%%%%%%%%%%%%%%%%%%%%%%%%%%%%%

%%%%%%%%%%%%%%%%%%%%%%%%%%%%%%%%%%%%%%%%%%%%%%%%%%%%%%%%%%%%%%%%%%%%%%%%%
\end{document}